\documentclass[10pt]{article}
\usepackage{amsmath}
\usepackage{amssymb}
\usepackage{amsthm}
\usepackage{mathrsfs}
\numberwithin{equation}{section}

\begin{document}

\title{Multilinear $\theta$-type Calder\'on--Zygmund operators and commutators on products of weighted Morrey spaces}
\author{Xia Han and Hua Wang \footnote{In memory of Li Xue. E-mail address: 1044381894@qq.com, wanghua@pku.edu.cn.}\\
\footnotesize{School of Mathematics and System Sciences, Xinjiang University, Urumqi 830046, P. R. China}}
\date{}
\maketitle
\begin{abstract}
In this paper, we consider the boundedness properties of multilinear $\theta$-type Calder\'on--Zygmund operators $T_\theta$ recently introduced in the literature. First, we prove strong type and weak type estimates for multilinear $\theta$-type Calder\'on--Zygmund operators on products of weighted Morrey spaces with multiple weights. Then we discuss strong type estimates for both multilinear commutators and iterated commutators of $T_\theta$ on products of these spaces with multiple weights. Furthermore, the weak end-point estimates for commutators of $T_\theta$ and pointwise multiplication with functions in bounded mean oscillation are established too.
\end{abstract}
\tableofcontents
\section{Introduction}
\label{intro}
In this paper, the symbols $\mathbb R$ and $\mathbb N$ stand for the sets of all real numbers and natural numbers, respectively. Let $\mathbb R^n$ be the $n$-dimensional Euclidean space with the Euclidean norm $|\cdot|$ and the Lebesgue measure $dx$. Let $m\in\mathbb N$ and $(\mathbb R^n)^m=\overbrace{\mathbb R^n\times\cdots\times\mathbb R^n}^m$ be the $m$-fold product space. We denote by $\mathscr S(\mathbb R^n)$ the space of all Schwartz functions on $\mathbb R^n$ and by $\mathscr S'(\mathbb R^n)$ its dual space, the set of all tempered distributions on $\mathbb R^n$.
Calder\'on--Zygmund singular integral operators and their generalizations on the Euclidean space $\mathbb R^n$ have been extensively studied (see \cite{duoand,garcia,grafakos2,stein2} for instance). In particular, Yabuta \cite{yabuta} introduced certain $\theta$-type Calder\'on--Zygmund operators to facilitate his study of certain classes of pseudo-differential operators. Following the terminology of Yabuta \cite{yabuta}, we introduce the so-called $\theta$-type Calder\'on--Zygmund operators as follows.

\newtheorem{defn}{Definition}[section]

\begin{defn}
Let $\theta$ be a nonnegative, nondecreasing function on $\mathbb R^+:=(0,+\infty)$ with $0<\theta(1)<+\infty$ and
\begin{equation*}
\int_0^1\frac{\theta(t)}{t}\,dt<+\infty.
\end{equation*}
A measurable function $K(x,y)$ on $\mathbb R^n\times\mathbb R^n\setminus\{(x,y):x=y\}$ is said to be a $\theta$-type Calder\'on--Zygmund kernel, if there exists a constant $A>0$ such that
\begin{enumerate}
  \item $\big|K(x,y)\big|\leq\frac{A}{|x-y|^{n}}$, \quad \mbox{for any}\, $x\neq y;$
  \item $\big|K(x,y)-K(z,y)\big|+\big|K(y,x)-K(y,z)\big|\leq \frac{A}{|x-y|^{n}}\cdot\theta\Big(\frac{|x-z|}{|x-y|}\Big)$,
  \quad \mbox{for}\, $|x-z|<\frac{|x-y|}{2}$.
\end{enumerate}
\end{defn}
\begin{defn}
Let ${\mathcal T}_\theta$ be a linear operator from $\mathscr S(\mathbb R^n)$ into its dual $\mathscr S'(\mathbb R^n)$. We say that ${\mathcal T}_\theta$ is a $\theta$-type Calder\'on--Zygmund operator with associated kernel $K$ if
\begin{enumerate}
  \item ${\mathcal T}_\theta$ can be extended to be a bounded linear operator on $L^2(\mathbb R^n);$
  \item for any $f\in C^\infty_0(\mathbb R^n)$ and for all $x\notin\mathrm{supp\,}f$, there is a $\theta$-type Calder\'on--Zygmund kernel $K(x,y)$ such that
\begin{equation*}
{\mathcal T}_\theta f(x):=\int_{\mathbb R^n}K(x,y)f(y)\,dy,
\end{equation*}
where $C^\infty_0(\mathbb R^n)$ is the space consisting of all infinitely differentiable functions on $\mathbb R^n$ that have compact support.
\end{enumerate}
\end{defn}
Note that the classical Calder\'on--Zygmund operator with standard kernel (see \cite{duoand,garcia}) is a special case of $\theta$-type operator ${\mathcal T}_{\theta}$ when $\theta(t)=t^{\delta}$ with $0<\delta\leq1$.

In 2009, Maldonado and Naibo \cite{ma} considered the bilinear $\theta$-type Calder\'on--Zygmund operators which are natural generalizations of the linear case, and established weighted norm inequalities for bilinear $\theta$-type Calder\'on--Zygmund operators on products of weighted Lebesgue spaces with Muckenhoupt weights. Moreover, they applied these operators to the study of certain paraproducts and bilinear pseudo-differential operators with mild regularity. Later, in 2014, Lu and Zhang \cite{lu} introduced the general $m$-linear $\theta$-type Calder\'on--Zygmund operators and their commutators for $m\geq2$, and established boundedness properties of these multilinear operators and multilinear commutators on products of weighted Lebesgue spaces with multiple weights. In addition, they gave some applications to the paraproducts and bilinear pseudo-differential operators with mild regularity and their commutators too. Following \cite{lu}, we now give the definition of the multilinear $\theta$-type Calder\'on--Zygmund operators.
\begin{defn}
Let $\theta$ be a nonnegative, nondecreasing function on $\mathbb R^+$ with $0<\theta(1)<+\infty$ and
\begin{equation}\label{theta1}
\int_0^1\frac{\theta(t)}{t}\,dt<+\infty.
\end{equation}
A measurable function $K(x,y_1,\ldots,y_m)$, defined away from the diagonal $x=y_1=\cdots=y_m$ in $(\mathbb R^n)^{m+1}$, is called an $m$-linear $\theta$-type Calder\'on--Zygmund kernel, if there exists a constant $A>0$ such that
\begin{enumerate}
  \item
\begin{equation}\label{size}
\big|K(x,y_1,\ldots,y_m)\big|\leq\frac{A}{(|x-y_1|+\cdots+|x-y_m|)^{mn}}
\end{equation}
for all $(x,y_1,\dots,y_m)\in(\mathbb R^n)^{m+1}$ with $x\neq y_k$ for some $k\in\{1,2,\dots,m\}$, and
  \item
\begin{equation}
\begin{split}
&\big|K(x,y_1,\dots,y_m)-K(x',y_1,\dots,y_m)\big|\\
\leq &\frac{A}{(|x-y_1|+\cdots+|x-y_m|)^{mn}}\cdot\theta\bigg(\frac{|x-x'|}{|x-y_1|+\cdots+|x-y_m|}\bigg)
\end{split}
\end{equation}
whenever $|x-x'|\leq\frac{\,1\,}{2}\max_{1\leq i\leq m}|x-y_i|$, and
  \item for each fixed $k$ with $1\leq k\leq m$,
\begin{equation}
\begin{split}
&\big|K(x,y_1,\dots,y_k,\dots,y_m)-K(x,y_1,\dots,y'_k,\dots,y_m)\big|\\
\leq & \frac{A}{(|x-y_1|+\cdots+|x-y_m|)^{mn}}\cdot\theta\bigg(\frac{|y_k-y'_k|}{|x-y_1|+\cdots+|x-y_m|}\bigg)
\end{split}
\end{equation}
whenever $|y_k-y'_k|\leq\frac{\,1\,}{2}\max_{1\leq i\leq m}|x-y_i|$.
\end{enumerate}
\end{defn}

\begin{defn}
Let $m\in\mathbb N$ and $T_{\theta}$ be an $m$-linear operator initially defined on the $m$-fold product of Schwartz spaces and taking values into the space of tempered distributions, i.e.,
\begin{equation*}
T_{\theta}:\overbrace{\mathscr S(\mathbb R^n)\times\cdots\times\mathscr S(\mathbb R^n)}^m\to\mathscr S'(\mathbb R^n).
\end{equation*}
We say that $T_\theta$ is an $m$-linear $\theta$-type Calder\'on--Zygmund operator if
\begin{enumerate}
  \item $T_\theta$ can be extended to be a bounded multilinear operator from $L^{q_1}(\mathbb R^n)\times\cdots\times L^{q_m}(\mathbb R^n)$ into $L^q(\mathbb R^n)$ for some $q_1,\ldots,q_m\in[1,+\infty)$ and $q\in[1/m,+\infty)$ with $1/q=\sum_{k=1}^m 1/{q_k};$
  \item for any given $m$-tuples $\vec{f}=(f_1,\ldots,f_m)$, there is an $m$-linear $\theta$-type Calder\'on--Zygmund kernel $K(x,y_1,\dots,y_m)$ such that
\begin{equation*}
\begin{split}
T_\theta(\vec{f})(x)&=T_\theta(f_1,\dots,f_m)(x)\\
&:=\int_{(\mathbb R^n)^m}K(x,y_1,\dots,y_m)f_1(y_1)\cdots f_m(y_m)\,dy_1\cdots dy_m
\end{split}
\end{equation*}
whenever $x\notin\bigcap_{k=1}^m\mathrm{supp\,}f_k$ and each $f_k\in C^\infty_0(\mathbb R^n)$ for $k=1,2,\dots,m$.
\end{enumerate}
\end{defn}

\newtheorem{theorem}{Theorem}[section]

\newtheorem{corollary}{Corollary}[section]

\newtheorem{rek}{Remark}[section]

\newtheorem{lemma}{Lemma}[section]

We note that, if we simply take $\theta(t)=t^{\varepsilon}$ for some $0<\varepsilon\leq1$, then the multilinear $\theta$-type operator $T_{\theta}$ is exactly the multilinear Calder\'on--Zygmund operator, which was systematically studied by many authors. There is a vast literature of results of this nature, pioneered by the work of Grafakos and Torres \cite{grafakos}, we refer the reader to \cite{grafakos3,lerner,perez3} and the references therein for more details. In 2014, the following weighted strong-type and weak-type estimates of multilinear $\theta$-type Calder\'on--Zygmund operators on products of weighted Lebesgue spaces were proved by Lu and Zhang in \cite{lu}.

\begin{theorem}[\cite{lu}]\label{strong}
Let $m\in\mathbb N$ and $T_{\theta}$ be an $m$-linear $\theta$-type Calder\'on--Zygmund operator with $\theta$ satisfying the condition \eqref{theta1}. If $p_1,\ldots,p_m\in(1,+\infty)$ and $p\in(1/m,+\infty)$ with $1/p=\sum_{k=1}^m 1/{p_k}$, and $\vec{w}=(w_1,\ldots,w_m)$ satisfies the multilinear $A_{\vec{P}}$ condition, then there exists a constant $C>0$ independent of $\vec{f}=(f_1,\ldots,f_m)$ such that
\begin{equation*}
\big\|T_{\theta}(\vec{f})\big\|_{L^p(\nu_{\vec{w}})}\le C\prod_{k=1}^m\big\|f_k\big\|_{L^{p_k}(w_k)},
\end{equation*}
where $\nu_{\vec{w}}=\prod_{k=1}^m w_k^{p/{p_k}}$.
\end{theorem}

\begin{theorem}[\cite{lu}]\label{weak}
Let $m\in\mathbb N$ and $T_{\theta}$ be an $m$-linear $\theta$-type Calder\'on--Zygmund operator with $\theta$ satisfying the condition \eqref{theta1}. If $p_1,\ldots,p_m\in[1,+\infty)$, $\min\{p_1,\ldots,p_m\}=1$ and $p\in[1/m,+\infty)$ with $1/p=\sum_{k=1}^m 1/{p_k}$, and $\vec{w}=(w_1,\ldots,w_m)$ satisfies the multilinear $A_{\vec{P}}$ condition, then there exists a constant $C>0$ independent of $\vec{f}=(f_1,\ldots,f_m)$ such that
\begin{equation*}
\big\|T_{\theta}(\vec{f})\big\|_{WL^p(\nu_{\vec{w}})}\le C\prod_{k=1}^m\big\|f_k\big\|_{L^{p_k}(w_k)},
\end{equation*}
where $\nu_{\vec{w}}=\prod_{k=1}^m w_k^{p/{p_k}}$.
\end{theorem}
For any given $p\in(0,+\infty)$ and $w$(weight function), the space $L^p(w)$ is defined as the set of all integrable functions $f$ on $\mathbb R^n$ such that
\begin{equation*}
\|f\|_{L^p(w)}:=\bigg(\int_{\mathbb R^n}|f(x)|^pw(x)\,dx\bigg)^{1/p}<+\infty,
\end{equation*}
and the weak space $WL^{p}(w)$ is defined as the set of all measurable functions $f$ on $\mathbb R^n$ such that
\begin{equation*}
\|f\|_{WL^p(w)}:=\sup_{\lambda>0}\lambda\cdot w\big(\big\{x\in\mathbb R^n:|f(x)|>\lambda\big\}\big)^{1/p}<+\infty,
\end{equation*}
where $w(E):=\int_{E}w(x)\,dx$ for a Lebesgue measurable set $E\subset\mathbb R^n$. When $w\equiv1$, we denote simply by $L^p(\mathbb R^n)$ and $WL^p(\mathbb R^n)$.

\begin{rek}
For the linear case $m=1$, the above weighted results were given by Quek and Yang in \cite{quek}. For the bilinear case $m=2$, Theorems \ref{strong} and \ref{weak} were proved by Maldonado and Naibo in \cite{ma} when some additional conditions imposed on $\theta$. And when $\theta(t)=t^{\varepsilon}$ for some $0<\varepsilon\leq1$, Theorems \ref{strong} and \ref{weak} were obtained by Lerner et al. \cite{lerner}.
\end{rek}

Next, we give the definition of the commutator for the multilinear $\theta$-type Calder\'on--Zygmund operator. Given a collection of locally integrable functions $\vec{b}=(b_1,\ldots,b_m)$, the $m$-linear commutator of $T_{\theta}$ with $\vec{b}$ is defined by
\begin{equation}\label{multicomm}
\begin{split}
\big[\Sigma\vec{b},T_\theta\big](\vec{f})(x)&=\big[\Sigma\vec{b},T_\theta\big](f_1,\ldots,f_m)(x)\\
&:=\sum_{k=1}^m\big[b_k,T_{\theta}\big]_{k}(f_1,\ldots,f_m)(x),
\end{split}
\end{equation}
where each term is the commutator of $b_k$ and $T_{\theta}$ in the $k$-th entry of $T_{\theta}$; that is,
\begin{equation*}
\begin{split}
&\big[b_k,T_{\theta}\big]_{k}(f_1,\ldots,f_m)(x)\\
&=b_k(x)\cdot T_{\theta}(f_1,\ldots,f_k,\dots,f_m)(x)-T_{\theta}(f_1,\dots,b_kf_k,\dots,f_m)(x).
\end{split}
\end{equation*}
Then, at a formal level
\begin{equation*}
\begin{split}
&\big[\Sigma\vec{b},T_\theta\big](\vec{f})(x)=\big[\Sigma\vec{b},T_\theta\big](f_1,\ldots,f_m)(x)\\
&=\int_{(\mathbb R^n)^m}\sum_{k=1}^m\big[b_k(x)-b_k(y_k)\big]K(x,y_1,\dots,y_m)f_1(y_1)\cdots f_m(y_m)\,dy_1\cdots dy_m.
\end{split}
\end{equation*}
Obviously, when $m=1$ in the above definition, this operator coincides with the linear commutator $[b,{\mathcal T}_{\theta}]$(see \cite{liu,zhang2}), which is defined by
\begin{equation*}
[b,{\mathcal T}_{\theta}](f):=b\cdot{\mathcal T}_{\theta}(f)-{\mathcal T}_{\theta}(bf).
\end{equation*}
Let us now recall the definition of the space of $\mathrm{BMO}(\mathbb R^n)$(see \cite{duoand,john}). A locally integrable function $b(x)$ is said to belong to $\mathrm{BMO}(\mathbb R^n)$ if it satisfies
\begin{equation*}
\|b\|_*:=\sup_{B}\frac{1}{|B|}\int_B|b(x)-b_B|\,dx<+\infty,
\end{equation*}
where the supremum is taken over all balls $B$ in $\mathbb R^n$, and $b_B$ stands for the average of $b$ over $B$, i.e.,
\begin{equation*}
b_B:=\frac{1}{|B|}\int_B b(y)\,dy.
\end{equation*}
In the multilinear setting, we say that $\vec{b}=(b_1,\ldots,b_m)\in \mathrm{BMO}^m$, if each $b_k\in \mathrm{BMO}(\mathbb R^n)$ for $k=1,2,\dots,m$. For convenience, we will use the following notation
\begin{equation*}
\big\|\vec{b}\big\|_{\mathrm{BMO}^m}:=\max_{1\leq k\leq m}\big\|b_k\big\|_{\ast},\quad\mbox{for} \; \vec{b}=(b_1,\ldots,b_m)\in \mathrm{BMO}^m.
\end{equation*}
In 2014, Lu and Zhang \cite{lu} also proved some weighted estimate and $L\log L$-type estimate for multilinear commutators $\big[\Sigma\vec{b},T_\theta\big]$ defined in \eqref{multicomm} under a stronger condition \eqref{theta2} assumed on $\theta$, if $\vec{b}\in \mathrm{BMO}^m$.

\begin{theorem}[\cite{lu}]\label{comm}
Let $m\in\mathbb N$ and $\big[\Sigma\vec{b},T_\theta\big]$ be the $m$-linear commutator generated by $\theta$-type Calder\'on--Zygmund operator $T_{\theta}$ and $\vec{b}=(b_1,\ldots,b_m)\in\mathrm{BMO}^m;$ let $\theta$ satisfy
\begin{equation}\label{theta2}
\int_0^1\frac{\theta(t)\cdot(1+|\log t|)}{t}\,dt<+\infty.
\end{equation}
If $p_1,\ldots,p_m\in(1,+\infty)$ and $p\in(1/m,+\infty)$ with $1/p=\sum_{k=1}^m 1/{p_k}$, and $\vec{w}=(w_1,\ldots,w_m)\in A_{\vec{P}}$, then there exists a constant $C>0$ independent of $\vec{b}$ and $\vec{f}=(f_1,\ldots,f_m)$ such that
\begin{equation*}
\big\|\big[\Sigma\vec{b},T_\theta\big](\vec{f})\big\|_{L^p(\nu_{\vec{w}})}\leq C\cdot\big\|\vec{b}\big\|_{\mathrm{BMO}^m}\prod_{k=1}^m\big\|f_k\big\|_{L^{p_k}(w_k)},
\end{equation*}
where $\nu_{\vec{w}}=\prod_{k=1}^m w_k^{p/{p_k}}$.
\end{theorem}

\begin{theorem}[\cite{lu}]\label{Wcomm}
Let $m\in\mathbb N$ and $\big[\Sigma\vec{b},T_\theta\big]$ be the $m$-linear commutator generated by $\theta$-type Calder\'on--Zygmund operator $T_{\theta}$ and $\vec{b}=(b_1,\ldots,b_m)\in\mathrm{BMO}^m;$ let $\theta$ satisfy the condition \eqref{theta2}. If $p_k=1$, $k=1,2,\dots,m$ and $\vec{w}=(w_1,\ldots,w_m)\in A_{(1,\dots,1)}$, then for any given $\lambda>0$, there exists a constant $C>0$ independent of $\vec{b}$, $\vec{f}=(f_1,\ldots,f_m)$ and $\lambda$ such that
\begin{equation*}
\begin{split}
&\nu_{\vec{w}}\Big(\Big\{x\in\mathbb R^n:\big|\big[\Sigma\vec{b},T_\theta\big](\vec{f})(x)\big|>\lambda^m\Big\}\Big)\\
&\leq C\cdot\Phi\big(\big\|\vec{b}\big\|_{\mathrm{BMO}^m}\big)^{1/m}
\prod_{k=1}^m\bigg(\int_{\mathbb R^n}\Phi\bigg(\frac{|f_k(x)|}{\lambda}\bigg)w_k(x)\,dx\bigg)^{1/m},
\end{split}
\end{equation*}
where $\nu_{\vec{w}}=\prod_{k=1}^m w_k^{1/{m}}$, $\Phi(t):=t\cdot(1+\log^+t)$ and $\log^+t:=\max\{\log t,0\}$.
\end{theorem}

\begin{rek}
As is well known, (multilinear) commutator has a greater degree of singularity than the underlying (multilinear) $\theta$-type operator, so more regular condition imposed on $\theta(t)$ is reasonable. Obviously, our condition \eqref{theta2} is slightly stronger than the condition \eqref{theta1}. For such type of commutators, the condition that $\theta(t)$ satisfying \eqref{theta2} is needed in the linear case (see \cite{liu,zhang2} for more details), so does in the multilinear case. Moreover, it is straightforward to check that when $\theta(t)=t^{\varepsilon}$ for some $\varepsilon>0$,
\begin{equation*}
\int_0^1\frac{t^{\varepsilon}\cdot(1+|\log t|)}{t}\,dt=\int_0^1 t^{\varepsilon-1}\cdot\bigg(1+\log\frac{\,1\,}{t}\bigg)dt<+\infty.
\end{equation*}
Thus, the multilinear Calder\'on--Zygmund operator is also the multilinear $\theta$-type operator $T_{\theta}$ with $\theta(t)$ satisfying \eqref{theta2}.
\end{rek}

\begin{rek}
When $m=1$, the above weighted endpoint estimate for the linear commutator $[b,\mathcal{T}_{\theta}]$ was given by Zhang and Xu in \cite{zhang2} (for the unweighted case, see \cite{liu}). Since $\mathcal{T}_{\theta}$ is bounded on $L^p(w)$ for $1<p<+\infty$ and $w\in A_p$ as mentioned earlier, then by the well-known boundedness criterion for commutators of linear operators, which was obtained by Alvarez et al. in \cite{alvarez}, we know that $[b,\mathcal{T}_{\theta}]$ is also bounded on $L^p(w)$ for all $1<p<+\infty$ and $w\in A_p$, whenever $b\in \mathrm{BMO}(\mathbb R^n)$.
\end{rek}

\begin{rek}
When $m\geq2$, $w_1=\cdots=w_m\equiv1$ and $\theta(t)=t^{\varepsilon}$ for some $\varepsilon>0$, P\'erez and Torres \cite{perez3} proved that if $\vec{b}=(b_1,\ldots,b_m)\in\mathrm{BMO}^m$, then
\begin{equation*}
\big[\Sigma\vec{b},T_\theta\big]:L^{p_1}(\mathbb R^n)\times\cdots\times L^{p_m}(\mathbb R^n)\to L^p(\mathbb R^n)
\end{equation*}
for $1<p_k<+\infty$ and $1<p<+\infty$ with $1/p=1/{p_1}+\cdots+1/{p_m}$, where $k=1,2,\dots,m$. And when $m\geq2$ and $\theta(t)=t^{\varepsilon}$ for some $\varepsilon>0$, Theorems \ref{comm} and \ref{Wcomm} were obtained by Lerner et al. in \cite{lerner}. Namely, Lerner et al.\cite{lerner} proved that if $\vec{b}=(b_1,\ldots,b_m)\in\mathrm{BMO}^m$ and $\vec{w}=(w_1,\ldots,w_m)\in A_{\vec{P}}$, then
\begin{equation*}
\big[\Sigma\vec{b},T_\theta\big]:L^{p_1}(w_1)\times\cdots\times L^{p_m}(w_m)\to L^p(\nu_{\vec{w}})
\end{equation*}
for $1<p_k<+\infty$ and $1/m<p<+\infty$ with $1/p=1/{p_1}+\cdots+1/{p_m}$, where $k=1,2,\dots,m$.
\end{rek}

\begin{rek}
We will give alternative proof of Theorem $\ref{comm}$, which shows that the conclusion of Theorem $\ref{comm}$ still holds provided that $\theta(t)$ only fulfills \eqref{theta1}, see the Appendix section for more details.
\end{rek}

Motivated by \cite{perez} and \cite{lu}, we will consider another type of commutators on $\mathbb R^n$. Assume that $\vec{b}=(b_1,\dots,b_m)$ is a collection of locally integrable functions, we define the iterated commutator $\big[\Pi\vec{b},T_\theta\big]$ as
\begin{equation*}
\begin{split}
\big[\Pi\vec{b},T_\theta\big](\vec{f})(x)&=\big[\Pi\vec{b},T_\theta\big](f_1,\ldots,f_m)(x)\\
&:=[b_1,[b_2,\dots[b_{m-1},[b_m,T_{\theta}]_m]_{m-1}\dots]_2]_1(f_1,\ldots,f_m)(x),
\end{split}
\end{equation*}
where
\begin{equation*}
\begin{split}
&\big[b_k,T_{\theta}\big]_{k}(f_1,\ldots,f_m)(x)\\
&=b_k(x)\cdot T_{\theta}(f_1,\ldots,f_k,\dots,f_m)(x)-T_{\theta}(f_1,\dots,b_kf_k,\dots,f_m)(x).
\end{split}
\end{equation*}
Then $\big[\Pi\vec{b},T_\theta\big]$ could be expressed in the following way
\begin{equation}\label{iteratedc}
\begin{split}
&\big[\Pi\vec{b},T_\theta\big](\vec{f})(x)=\big[\Pi\vec{b},T_\theta\big](f_1,\ldots,f_m)(x)\\
&=\int_{(\mathbb R^n)^m}\prod_{k=1}^m\big[b_k(x)-b_k(y_k)\big]K(x,y_1,\dots,y_m)f_1(y_1)\cdots f_m(y_m)\,dy_1\cdots dy_m.
\end{split}
\end{equation}
Following the arguments used in \cite{perez} and \cite{lu} with some minor modifications, we can also establish the corresponding results (strong type and weak endpoint estimates) for iterated commutators of multilinear $\theta$-type Calder\'on--Zygmund operators.

\begin{theorem}\label{commwh}
Let $m\in\mathbb N$ and $\big[\Pi\vec{b},T_\theta\big]$ be the iterated commutator generated by $\theta$-type Calder\'on--Zygmund operator $T_{\theta}$ and $\vec{b}=(b_1,\ldots,b_m)\in\mathrm{BMO}^m;$ let $\theta$ satisfy the condition \eqref{theta1}.
If $p_1,\ldots,p_m\in(1,+\infty)$ and $p\in(1/m,+\infty)$ with $1/p=\sum_{k=1}^m 1/{p_k}$, and $\vec{w}=(w_1,\ldots,w_m)\in A_{\vec{P}}$, then there exists a constant $C>0$ independent of $\vec{b}$ and $\vec{f}=(f_1,\ldots,f_m)$ such that
\begin{equation*}
\big\|\big[\Pi\vec{b},T_\theta\big](\vec{f})\big\|_{L^p(\nu_{\vec{w}})}\leq C\cdot\prod_{k=1}^m\big\|b_k\big\|_{*}\prod_{k=1}^m\big\|f_k\big\|_{L^{p_k}(w_k)},
\end{equation*}
where $\nu_{\vec{w}}=\prod_{k=1}^m w_k^{p/{p_k}}$.
\end{theorem}

\begin{theorem}\label{Wcommwh}
Let $m\in\mathbb N$ and $\big[\Pi\vec{b},T_\theta\big]$ be the iterated commutator generated by $\theta$-type Calder\'on--Zygmund operator $T_{\theta}$ and $\vec{b}=(b_1,\ldots,b_m)\in\mathrm{BMO}^m;$ let $\theta$ satisfy
\begin{equation}\label{theta3}
\int_0^1\frac{\theta(t)\cdot(1+|\log t|^m)}{t}\,dt<+\infty.
\end{equation}
If $p_k=1$, $k=1,2,\dots,m$ and $\vec{w}=(w_1,\ldots,w_m)\in A_{(1,\dots,1)}$, then for any given $\lambda>0$, there exists a constant $C>0$ independent of $\vec{f}=(f_1,\ldots,f_m)$ and $\lambda$ such that
\begin{equation*}
\begin{split}
&\nu_{\vec{w}}\Big(\Big\{x\in\mathbb R^n:\big|\big[\Pi\vec{b},T_\theta\big](\vec{f})(x)\big|>\lambda^m\Big\}\Big)\\
&\leq C\cdot
\prod_{k=1}^m\bigg(\int_{\mathbb R^n}\Phi^{(m)}\bigg(\frac{|f_k(x)|}{\lambda}\bigg)w_k(x)\,dx\bigg)^{1/m},
\end{split}
\end{equation*}
where $\nu_{\vec{w}}=\prod_{k=1}^m w_k^{1/{m}}$, $\Phi(t)=t\cdot(1+\log^+t)$ and $\Phi^{(m)}:=\overbrace{\Phi\circ\cdots\circ\Phi}^m$.
\end{theorem}
\begin{rek}
It was proved in \cite{perez} that when $\theta(t)=t^{\varepsilon}$ for some $\varepsilon>0$, the estimate in Theorem \ref{Wcommwh} is sharp in the sense that $\Phi^{(m)}$ cannot be replaced by $\Phi^{(k)}$ for any $k<m$.
\end{rek}
On the other hand, the classical Morrey spaces $L^{p,\kappa}(\mathbb R^n)$ were originally introduced by Morrey in \cite{morrey} to study the local regularity of solutions to second order elliptic partial differential equations. Nowadays these spaces have been studied intensively in the literature, and found a wide range of applications in harmonic analysis, potential theory and nonlinear dispersive equations. In 2009, Komori and Shirai \cite{komori} defined and investigated the weighted Morrey spaces $L^{p,\kappa}(w)$ for $1\leq p<+\infty$, which could be viewed as an extension of weighted Lebesgue spaces, and obtained the boundedness of some classical integral operators on these weighted spaces. In order to deal with the multilinear case $m\geq2$, we consider the weighted Morrey spaces $L^{p,\kappa}(w)$ here for all $0<p<+\infty$. We will extend the results obtained in \cite{lu} for $m$-linear $\theta$-type Calder\'on--Zygmund operators to the product of weighted Morrey spaces with multiple weights. Moreover, the corresponding weighted estimates for both multilinear commutators and iterated commutators are also considered. Let us first recall the definition of the spaces $L^{p,\kappa}(w)$ and $WL^{p,\kappa}(w)$.
\begin{defn}[\cite{komori}]\label{amalgam}
Let $0<p<+\infty$, $0\leq\kappa<1$ and let $w$ be a weight on $\mathbb R^n$. The weighted Morrey space $L^{p,\kappa}(w)$ is defined to be the set of all locally integrable functions $f$ on $\mathbb R^n$ satisfying
\begin{equation*}
\|f\|_{L^{p,\kappa}(w)}:=\sup_{B}\bigg(\frac{1}{w(B)^{\kappa}}\int_B|f(x)|^pw(x)\,dx\bigg)^{1/p}<+\infty,
\end{equation*}
where the supremum is taken over all balls $B$ in $\mathbb R^n$.
\end{defn}

\begin{defn}[\cite{komori}]\label{Wamalgam}
Let $0<p<+\infty$, $0\leq\kappa<1$ and let $w$ be a weight on $\mathbb R^n$. The weighted weak Morrey space $WL^{p,\kappa}(w)$ is defined to be the set of all measurable functions $f$ on $\mathbb R^n$ satisfying
\begin{equation*}
\begin{split}
\|f\|_{WL^{p,\kappa}(w)}:=&\sup_{B}\frac{1}{m(B)^{\kappa/p}}\sup_{\lambda>0}\lambda\cdot w\big(\big\{x\in B:|f(x)|>\lambda\big\}\big)^{1/p}<+\infty,
\end{split}
\end{equation*}
where the supremum is taken over all balls $B$ in $\mathbb R^n$ and all $\lambda>0$.
\end{defn}
Note that when $w\in\Delta_2$, then $L^{p,0}(w)=L^p(w)$, $WL^{p,0}(w)=WL^p(w)$ and $L^{p,1}(w)=L^{\infty}(w)$ by the Lebesgue differentiation
theorem with respect to $w$.

In order to deal with the end-point case of the commutators, we have to consider the following $L\log L$-type space, which was introduced by the second author in \cite{wang3,wang4} (for the unweighted case, see also \cite{lida} and \cite{sa}).
\begin{defn}
Let $p=1$, $0\leq\kappa<1$ and let $w$ be a weight on $\mathbb R^n$. We denote by $(L\log L)^{1,\kappa}(w)$ the weighted Morrey space of $L\log L$ type, the space of all locally integrable functions $f$ defined on $\mathbb R^n$ with finite norm
$\big\|f\big\|_{(L\log L)^{1,\kappa}(w)}$.
\begin{equation*}
(L\log L)^{1,\kappa}(w):=\Big\{f:\big\|f\big\|_{(L\log L)^{1,\kappa}(w)}<\infty\Big\},
\end{equation*}
where
\begin{equation*}
\begin{split}
\big\|f\big\|_{(L\log L)^{1,\kappa}(w)}:=&\sup_{B}w(B)^{1-\kappa}\big\|f\big\|_{L\log L(w),B}.
\end{split}
\end{equation*}
\end{defn}
Here $\|\cdot\|_{L\log L(w),B}$ denotes the weighted Luxemburg norm, whose definition will be given in Section \ref{sec3} below. Note that $t\leq t\cdot(1+\log^+t)$ for any $t>0$. By definition, for any ball $B$ in $\mathbb R^n$ and $w\in A_\infty$, then we have
\begin{equation*}
\big\|f\big\|_{L(w),B}\leq\big\|f\big\|_{L\log L(w),B},
\end{equation*}
which means that the following inequality (it can be viewed as a generalized Jensen's inequality)
\begin{equation}\label{main esti1}
\big\|f\big\|_{L(w),B}=\frac{1}{w(B)}\int_{B}|f(x)|w(x)\,dx\leq\big\|f\big\|_{L\log L(w),B}
\end{equation}
holds for any ball $B\subset\mathbb R^n$. Hence, for all $0<\kappa<1$ and $w\in A_\infty$, we can further obtain the following inclusion
from \eqref{main esti1}:
\begin{equation*}
(L\log L)^{1,\kappa}(w)\hookrightarrow L^{1,\kappa}(w).
\end{equation*}
It is known that $L^{p,\kappa}$ is an extension of $L^p$ in the sense that $L^{p,0}=L^p$.
Motivated by the works in \cite{komori,lu,ma}, the main purpose of this paper is to establish boundedness properties of multilinear $\theta$-type Calder\'on--Zygmund operators and their commutators on products of weighted Morrey spaces with multiple weights.

In what follows, the letter $C$ always stands for a positive constant independent of the main parameters and not necessarily the same at each occurrence. The symbol $\mathbf{X}\lesssim \mathbf{Y}$ means that there is a constant $C>0$ such that $\mathbf{X}\leq C\mathbf{Y}$.
The symbol $\mathbf{X}\approx\mathbf{Y}$ means that there is a constant $C>0$ such that $C^{-1}\mathbf{Y}\leq\mathbf{X}\leq C\mathbf{Y}$.

\section{Main results}
Our first two results on the boundedness properties of multilinear $\theta$-type Calder\'on--Zygmund operators can be formulated as follows.

\begin{theorem}\label{mainthm:1}
Let $m\geq 2$ and $T_{\theta}$ be an $m$-linear $\theta$-type Calder\'on--Zygmund operator with $\theta$ satisfying the condition \eqref{theta1}. If $1<p_1,\dots,p_m<+\infty$ and $1/m<p<+\infty$ with $1/p=\sum_{i=1}^m 1/{p_i}$, and $\vec{w}=(w_1,\ldots,w_m)\in A_{\vec{P}}$ with $w_1,\ldots,w_m\in A_\infty$, then for any $0<\kappa<1$, the multilinear operator $T_{\theta}$ is bounded from $L^{p_1,\kappa}(w_1)\times L^{p_2,\kappa}(w_2)\times\cdots
\times L^{p_m,\kappa}(w_m)$ into $L^{p,\kappa}(\nu_{\vec{w}})$ with $\nu_{\vec{w}}=\prod_{i=1}^m w_i^{p/{p_i}}$.
\end{theorem}

\begin{theorem}\label{mainthm:2}
Let $m\geq 2$ and $T_{\theta}$ be an $m$-linear $\theta$-type Calder\'on--Zygmund operator with $\theta$ satisfying the condition \eqref{theta1}. If $1\leq p_1,\dots,p_m<+\infty$, $\min\{p_1,\ldots,p_m\}=1$ and $1/m\leq p<+\infty$ with $1/p=\sum_{i=1}^m 1/{p_i}$, and $\vec{w}=(w_1,\ldots,w_m)\in A_{\vec{P}}$ with $w_1,\ldots,w_m\in A_\infty$, then for any $0<\kappa<1$, the multilinear operator $T_{\theta}$ is bounded from $L^{p_1,\kappa}(w_1)\times L^{p_2,\kappa}(w_2)\times\cdots
\times L^{p_m,\kappa}(w_m)$ into $WL^{p,\kappa}(\nu_{\vec{w}})$ with $\nu_{\vec{w}}=\prod_{i=1}^m w_i^{p/{p_i}}$.
\end{theorem}

Our next theorem concerns norm inequalities for the multilinear commutator $\big[\Sigma\vec{b},T_\theta\big]$ with $\vec{b}\in\mathrm{BMO}^m$.

\begin{theorem}\label{mainthm:3}
Let $m\geq2$ and $\big[\Sigma\vec{b},T_\theta\big]$ be the $m$-linear commutator of $\theta$-type Calder\'on--Zygmund operator $T_{\theta}$ with $\theta$ satisfying the condition \eqref{theta1} and $\vec{b}\in \mathrm{BMO}^m$. If $1<p_1,\dots,p_m<+\infty$ and $1/m<p<+\infty$ with $1/p=\sum_{i=1}^m 1/{p_i}$, and $\vec{w}=(w_1,\ldots,w_m)\in A_{\vec{P}}$ with $w_1,\ldots,w_m\in A_\infty$, then for any $0<\kappa<1$, the multilinear commutator $\big[\Sigma\vec{b},T_\theta\big]$ is bounded from $L^{p_1,\kappa}(w_1)\times L^{p_2,\kappa}(w_2)\times\cdots
\times L^{p_m,\kappa}(w_m)$ into $L^{p,\kappa}(\nu_{\vec{w}})$ with $\nu_{\vec{w}}=\prod_{i=1}^m w_i^{p/{p_i}}$.
\end{theorem}

For the endpoint case $p_1=p_2=\cdots=p_m=1$, we will also prove the following weak-type $L\log L$ estimate for the multilinear commutator $\big[\Sigma\vec{b},T_\theta\big]$ in the weighted Morrey spaces with multiple weights.

\begin{theorem}\label{mainthm:4}
Let $m\geq2$ and $\big[\Sigma\vec{b},T_\theta\big]$ be the $m$-linear commutator of $\theta$-type Calder\'on--Zygmund operator $T_{\theta}$ with $\theta$ satisfying the condition \eqref{theta2} and $\vec{b}\in \mathrm{BMO}^m$. Assume that $\vec{w}=(w_1,\ldots,w_m)\in A_{(1,\dots,1)}$ with $w_1,\ldots,w_m\in A_\infty$. If $p_i=1$, $i=1,2,\ldots,m$ and $p=1/m$, then for any given $\lambda>0$ and any ball $B\subset\mathbb R^n$, there exists a constant $C>0$ such that
\begin{equation*}
\begin{split}
&\frac{1}{\nu_{\vec{w}}(B)^{m\kappa}}\cdot\Big[\nu_{\vec{w}}\Big(\Big\{x\in B:\big|\big[\Sigma\vec{b},T_\theta\big](\vec{f})(x)\big|>\lambda^m\Big\}\Big)\Big]^m\\
&\leq C\cdot\Phi\big(\big\|\vec{b}\big\|_{\mathrm{BMO}^m}\big)
\prod_{i=1}^m\bigg\|\Phi\bigg(\frac{|f_i|}{\lambda}\bigg)\bigg\|_{(L\log L)^{1,\kappa}(w_i)},
\end{split}
\end{equation*}
where $\nu_{\vec{w}}=\prod_{i=1}^m w_i^{1/{m}}$ and $\Phi(t)=t\cdot(1+\log^+t)$.
\end{theorem}

\begin{rek}
From the above definitions and Theorem \ref{mainthm:4}, we can roughly say that the multilinear commutator $\big[\Sigma\vec{b},T_\theta\big]$ is bounded from $(L\log L)^{1,\kappa}(w_1)\times(L\log L)^{1,\kappa}(w_2)\times\cdots
\times(L\log L)^{1,\kappa}(w_m)$ into $WL^{1/m,\kappa}(\nu_{\vec{w}})$ with $\nu_{\vec{w}}=\prod_{i=1}^m w_i^{1/{m}}$.
\end{rek}

\section{Notations and preliminaries}
\label{sec3}
\subsection{Multiple weights}
For any $r>0$ and $x\in\mathbb R^n$, let $B(x,r)=\big\{y\in\mathbb R^n:|x-y|<r\big\}$ denote the open ball centered at $x$ with radius $r$, $B(x,r)^{\complement}=\mathbb R^n\backslash B(x,r)$ denote its complement and $|B(x,r)|$ be the Lebesgue measure of the ball $B(x,r)$. We also use the notation $\chi_{B(x,r)}$ to denote the characteristic function of $B(x,r)$. For some $t>0$, the notation $tB$ stands for the ball with the same center as $B$ whose radius is $t$ times that of $B$.

A weight $w$ is said to belong to the Muckenhoupt class $A_p$ for $1<p<+\infty$, if there exists a constant $C>0$ such that
\begin{equation*}
\bigg(\frac1{|B|}\int_B w(x)\,dx\bigg)^{1/p}\bigg(\frac1{|B|}\int_B w(x)^{-p'/p}\,dx\bigg)^{1/{p'}}\leq C
\end{equation*}
for every ball $B$ in $\mathbb R^n$, where $p'$ is the conjugate exponent of $p$ such that $1/p+1/{p'}=1$. The class $A_1$ is defined replacing the above inequality by
\begin{equation*}
\frac1{|B|}\int_B w(x)\,dx\leq C\cdot\underset{x\in B}{\mbox{ess\,inf}}\;w(x)
\end{equation*}
for every ball $B$ in $\mathbb R^n$. Since the $A_p$ classes are increasing with respect to $p$, the $A_\infty$ class of weights is defined in a natural way by
\begin{equation*}
A_\infty:=\bigcup_{1\leq p<+\infty}A_p.
\end{equation*}
Moreover, the following characterization will often be used in the sequel. There are positive constants $C$ and $\delta$ such that for any ball $B$ and any measurable set $E$ contained in $B$,
\begin{equation}\label{compare}
\frac{w(E)}{w(B)}\leq C\bigg(\frac{|E|}{|B|}\bigg)^\delta.
\end{equation}
Given a Lebesgue measurable set $E$, we denote the characteristic function of $E$ by $\chi_E$. We say that a weight $w$ satisfies the doubling condition, simply denoted by $w\in\Delta_2$, if there is an absolute constant $C>0$ such that
\begin{equation}\label{weights}
w(2B)\leq C\,w(B)
\end{equation}
holds for any ball $B$ in $\mathbb R^n$. If $w\in A_p$ with $1\leq p<+\infty$ (or $w\in A_\infty$), then we have that $w\in\Delta_2$.

Recently, the theory of multiple weights adapted to multilinear Calder\'on--Zygmund operators was developed by Lerner et al. in \cite{lerner}. New more refined multilinear maximal function was defined and used in \cite{lerner} to characterize the class of multiple $A_{\vec{P}}$ weights, and to obtain some weighted estimates for multilinear Calder\'on--Zygmund operators. Now let us recall the definition of multiple weights. For $m$ exponents $p_1,\ldots,p_m\in[1,+\infty)$, we will often write $\vec{P}$ for the vector $\vec{P}=(p_1,\ldots,p_m)$, and $p$ for the number given by $1/p=\sum_{k=1}^m 1/{p_k}$ with $p\in[1/m,+\infty)$. Given $\vec{w}=(w_1,\ldots,w_m)$, let us set
\begin{equation*}
\nu_{\vec{w}}=\prod_{k=1}^m w_k^{p/{p_k}}.
\end{equation*}
We say that $\vec{w}$ satisfies the multilinear $A_{\vec{P}}$ condition if it satisfies
\begin{equation}\label{multiweight}
\sup_B\bigg(\frac{1}{|B|}\int_B \nu_{\vec{w}}(x)\,dx\bigg)^{1/p}\prod_{k=1}^m\bigg(\frac{1}{|B|}\int_B w_k(x)^{-p'_k/{p_k}}\,dx\bigg)^{1/{p'_k}}<+\infty.
\end{equation}
When $p_k=1$ for some $k\in\{1,2,\dots,m\}$, the condition $\big(\frac{1}{|B|}\int_B w_k(x)^{-p'_k/{p_k}}\,dx\big)^{1/{p'_k}}$ is understood as $\big(\inf_{x\in B}w_k(x)\big)^{-1}$. In particular, when each $p_k=1$, $k=1,2,\dots,m$, we denote $A_{\vec{1}}=A_{(1,\dots,1)}$. One can easily check that $A_{(1,\dots,1)}$ is contained in $A_{\vec{P}}$ for each $\vec{P}$, however, the classes $A_{\vec{P}}$ are NOT increasing with the natural partial order (see \cite[Remark 7.3]{lerner}). It was shown in \cite{lerner} that these are the largest classes of weights for which all multilinear Calder\'on--Zygmund operators are bounded on weighted Lebesgue spaces. Moreover, in general, the condition $\vec{w}\in A_{\vec{P}}$ does not imply $w_k\in L^1_{\mathrm{loc}}(\mathbb R^n)$ for any $1\leq k\leq m$ (see \cite[Remark 7.2]{lerner}), but instead
\begin{lemma}[\cite{lerner}]\label{multi}
Let $p_1,\ldots,p_m\in[1,+\infty)$ and $1/p=\sum_{k=1}^m 1/{p_k}$. Then $\vec{w}=(w_1,\ldots,w_m)\in A_{\vec{P}}$ if and only if
\begin{equation}\label{multi2}
\left\{
\begin{aligned}
&\nu_{\vec{w}}\in A_{mp},\\
&w_k^{1-p'_k}\in A_{mp'_k},\quad k=1,\ldots,m,
\end{aligned}\right.
\end{equation}
where $\nu_{\vec{w}}=\prod_{k=1}^m w_k^{p/{p_k}}$ and the condition $w_k^{1-p'_k}\in A_{mp'_k}$ in the case $p_k=1$ is understood as $w_k^{1/m}\in A_1$.
\end{lemma}
Observe that in the linear case $m=1$ both conditions included in \eqref{multi2} represent the same $A_p$ condition. However, in the multilinear case $m\geq2$ neither of the conditions in \eqref{multi2} implies the other. We refer the reader to \cite{lerner} for further details.

\subsection{Orlicz spaces and Luxemburg norms}
Next we recall some basic definitions and facts from the theory of Orlicz spaces. For more information about these spaces the reader may consult the book \cite{rao}. Let $\mathcal A:[0,+\infty)\rightarrow[0,+\infty)$ be a Young function. That is, a continuous, convex and strictly increasing function with $\mathcal A(0)=0$ and such that $\mathcal A(t)\to +\infty$ as $t\to +\infty$. Given a Young function $\mathcal A$ and a ball $B$ in $\mathbb R^n$, we consider the $\mathcal A$-average of a function $f$ over a ball $B$, which is given by the following Luxemburg norm:
\begin{equation*}
\big\|f\big\|_{\mathcal A,B}
:=\inf\bigg\{\lambda>0:\frac{1}{|B|}\int_B\mathcal{A}\bigg(\frac{|f(x)|}{\lambda}\bigg)dx\leq1\bigg\}.
\end{equation*}
When $\mathcal A(t)=t^p$ with $1\leq p<+\infty$, it is easy to see that
\begin{equation*}
\big\|f\big\|_{\mathcal A,B}=\bigg(\frac{1}{|B|}\int_B\big|f(x)\big|^p\,dx\bigg)^{1/p};
\end{equation*}
that is, the Luxemburg norm coincides with the normalized $L^p$ norm. Associated to each Young function $\mathcal A$, one can define its complementary function $\bar{\mathcal A}$ by
\begin{equation*}
\bar{\mathcal A}(s):=\sup_{0\leq t<+\infty}\big[st-\mathcal A(t)\big], \quad 0\leq s<+\infty.
\end{equation*}
It is not difficult to check that such $\bar{\mathcal A}$ is also a Young function. A standard computation shows that for all $t>0$,
\begin{equation*}
t\leq \mathcal{A}^{-1}(t)\bar{\mathcal A}^{-1}(t)\leq 2t.
\end{equation*}
From this, it follows that the following generalized H\"older's inequality in Orlicz spaces holds for any given ball $B$ in $\mathbb R^n$.
\begin{equation*}
\frac{1}{|B|}\int_B\big|f(x)\cdot g(x)\big|\,dx\leq 2\big\|f\big\|_{\mathcal A,B}\big\|g\big\|_{\bar{\mathcal A},B}.
\end{equation*}
A particular case of interest, and especially in this paper, is the Young function $\Phi(t)=t\cdot(1+\log^+t)$, and we know that its complementary Young function is given by $\bar{\Phi}(t)\approx\exp(t)-1$. The corresponding averages will be denoted by
\begin{equation*}
\big\|f\big\|_{\Phi,B}=\big\|f\big\|_{L\log L,B} \qquad\mbox{and}\qquad
\big\|g\big\|_{\bar{\Phi},B}=\big\|g\big\|_{\exp L,B}.
\end{equation*}
Consequently, from the above generalized H\"older's inequality in Orlicz spaces, we also get
\begin{equation}\label{holder}
\frac{1}{|B|}\int_B\big|f(x)\cdot g(x)\big|\,dx\leq 2\big\|f\big\|_{L\log L,B}\big\|g\big\|_{\exp L,B}.
\end{equation}
To obtain endpoint weak-type estimates for the multilinear and iterated commutators on the product of weighted Morrey spaces, we need to define the $\mathcal A$-average of a function $f$ over a ball $B$ by means of the weighted Luxemburg norm; that is, given a Young function $\mathcal A$ and $w\in A_\infty$, we define (see \cite{rao,zhang})
\begin{equation*}
\big\|f\big\|_{\mathcal A(w),B}:=\inf\bigg\{\sigma>0:\frac{1}{w(B)}
\int_B\mathcal A\bigg(\frac{|f(x)|}{\sigma}\bigg)\cdot w(x)\,dx\leq1\bigg\}.
\end{equation*}
When $\mathcal A(t)=t$, this norm is denoted by $\|\cdot\|_{L(w),B}$, when $\Phi(t)=t\cdot(1+\log^+t)$, this norm is also denoted by $\|\cdot\|_{L\log L(w),B}$. The complementary Young function of $\Phi(t)$ is $\bar{\Phi}(t)\approx\exp(t)-1$ with the corresponding Luxemburg norm denoted by $\|\cdot\|_{\exp L(w),B}$. For $w\in A_\infty$ and for every ball $B$ in $\mathbb R^n$, we can also show the weighted version of \eqref{holder}. Namely, the following generalized H\"older's inequality in the weighted context is true for $f,g$ (see \cite{zhang} for instance).
\begin{equation}\label{Wholder}
\frac{1}{w(B)}\int_B\big|f(x)\cdot g(x)\big|w(x)\,dx\leq C\big\|f\big\|_{L\log L(w),B}\big\|g\big\|_{\exp L(w),B}.
\end{equation}
This estimate will play an important role in the proof of Theorem \ref{mainthm:4}.

\section{Proofs of Theorems \ref{mainthm:1} and \ref{mainthm:2}}

This section is concerned with the proofs of Theorems \ref{mainthm:1} and \ref{mainthm:2}. Before proving the main theorems of this section, we first state the following important results without proof (see \cite{duoand} and \cite{grafakos2}).
\begin{lemma}[\cite{grafakos2}]\label{Min}
Let $\big\{f_k\big\}_{k=1}^N$ be a sequence of $L^p(\nu)$ functions with $0<p<+\infty$ and $\nu\in A_\infty$. Then we have
\begin{equation*}
\Big\|\sum_{k=1}^N f_k\Big\|_{L^p(\nu)}\leq\mathcal{C}(p,N)\sum_{k=1}^N\big\|f_k\big\|_{L^p(\nu)},
\end{equation*}
where $\mathcal{C}(p,N)=\max\big\{1,N^{\frac{1-p}{p}}\big\}$. More specifically, $\mathcal{C}(p,N)=1$ for $1\leq p<+\infty$, and $\mathcal{C}(p,N)=N^{\frac{1-p}{p}}$ for $0<p<1$.
\end{lemma}

\begin{lemma}[\cite{grafakos2}]\label{WMin}
Let $\big\{f_k\big\}_{k=1}^N$ be a sequence of $WL^p(\nu)$ functions with $0<p<+\infty$ and $\nu\in A_\infty$. Then we have
\begin{equation*}
\Big\|\sum_{k=1}^N f_k\Big\|_{WL^p(\nu)}\leq\mathcal{C}'(p,N)\sum_{k=1}^N\big\|f_k\big\|_{WL^p(\nu)},
\end{equation*}
where $\mathcal{C}'(p,N)=\max\big\{N,N^{\frac{\,1\,}{p}}\big\}$. More specifically, $\mathcal{C}'(p,N)=N$ for $1\leq p<+\infty$, and $\mathcal{C}'(p,N)=N^{\frac{\,1\,}{p}}$ for $0<p<1$.
\end{lemma}

\begin{lemma}[\cite{duoand}]\label{Ainfty}
Let $w\in A_\infty$. Then for any ball $B$ in $\mathbb R^n$, the following reverse Jensen's inequality holds.
\begin{equation*}
\int_B w(x)\,dx\leq C|B|\cdot\exp\bigg(\frac{1}{|B|}\int_B\log w(x)\,dx\bigg).
\end{equation*}
\end{lemma}

We are now in a position to prove Theorems $\ref{mainthm:1}$ and $\ref{mainthm:2}$.
\begin{proof}[Proof of Theorem $\ref{mainthm:1}$]
Let $1<p_1,\dots,p_m<+\infty$ and $\vec{f}=(f_1,\dots,f_m)$ be in $L^{p_1,\kappa}(w_1)\times\cdots
\times L^{p_m,\kappa}(w_m)$ with $\vec{w}=(w_1,\dots,w_m)\in A_{\vec{P}}$ and $0<\kappa<1$. For any given ball $B$ in $\mathbb R^n$(denote by $x_0$ the center of $B$, and $r>0$ the radius of $B$), it is enough for us to show that
\begin{equation}\label{wangh1}
\frac{1}{\nu_{\vec{w}}(B)^{\kappa/p}}\bigg(\int_{B}\big|T_\theta(f_1,\dots,f_m)(x)\big|^p\nu_{\vec{w}}(x)\,dx\bigg)^{1/p}
\lesssim\prod_{i=1}^m\big\|f_i\big\|_{L^{p_i,\kappa}(w_i)}.
\end{equation}
To this end, for any $1\leq i\leq m$, we represent $f_i$ as
\begin{equation*}
f_i=f_i\cdot\chi_{2B}+f_i\cdot\chi_{(2B)^{\complement}}:=f^0_i+f^{\infty}_i;
\end{equation*}
and $2B=B(x_0,2r)$. Then we write
\begin{equation*}
\begin{split}
\prod_{i=1}^m f_i(y_i)&=\prod_{i=1}^m\Big(f^0_i(y_i)+f^{\infty}_i(y_i)\Big)\\
&=\sum_{\beta_1,\ldots,\beta_m\in\{0,\infty\}}f^{\beta_1}_1(y_1)\cdots f^{\beta_m}_m(y_m)\\
&=\prod_{i=1}^m f^0_i(y_i)+\sum_{(\beta_1,\dots,\beta_m)\in\mathfrak{L}}f^{\beta_1}_1(y_1)\cdots f^{\beta_m}_m(y_m),
\end{split}
\end{equation*}
where
\begin{equation*}
\mathfrak{L}:=\big\{(\beta_1,\dots,\beta_m):\beta_k\in\{0,\infty\},\mbox{there is at least one $\beta_k\neq0$},1\leq k\leq m\big\};
\end{equation*}
that is, each term of $\sum$ contains at least one $\beta_k\neq0$. Since $T_{\theta}$ is an $m$-linear operator, then by Lemma \ref{Min}($N=2^m$), we have
\begin{align}\label{I}
&\frac{1}{\nu_{\vec{w}}(B)^{\kappa/p}}\bigg(\int_{B}\big|T_\theta(f_1,\dots,f_m)(x)\big|^p\nu_{\vec{w}}(x)\,dx\bigg)^{1/p}\notag\\
&\leq\frac{C}{\nu_{\vec{w}}(B)^{\kappa/p}}
\bigg(\int_{B}\big|T_\theta(f^0_1,\dots,f^0_m)(x)\big|^p\nu_{\vec{w}}(x)\,dx\bigg)^{1/p}\notag\\
&+\sum_{(\beta_1,\dots,\beta_m)\in\mathfrak{L}}\frac{C}{\nu_{\vec{w}}(B)^{\kappa/p}}
\bigg(\int_{B}\big|T_\theta(f^{\beta_1}_1,\ldots,f^{\beta_m}_m)(x)\big|^p\nu_{\vec{w}}(x)\,dx\bigg)^{1/p}\notag\\
&:=I^{0,\dots,0}+\sum_{(\beta_1,\dots,\beta_m)\in\mathfrak{L}} I^{\beta_1,\dots,\beta_m}.
\end{align}
By the weighted strong-type estimate of $T_{\theta}$ (see Theorem \ref{strong}), we have
\begin{align}\label{I0}
I^{0,\dots,0}&\leq C\cdot\frac{1}{\nu_{\vec{w}}(B)^{\kappa/p}}
\prod_{i=1}^m\bigg(\int_{2B}|f_i(x)|^{p_i}w_i(x)\,dx\bigg)^{1/{p_i}}\notag\\
&\leq C\prod_{i=1}^m\big\|f_i\big\|_{L^{p_i,\kappa}(w_i)}\cdot
\frac{1}{\nu_{\vec{w}}(B)^{\kappa/p}}\prod_{i=1}^m w_i(2B)^{\kappa/{p_i}}.
\end{align}
Let $p_1,\ldots,p_m\in[1,+\infty)$ and $p\in[1/m,+\infty)$ with $1/p=\sum_{i=1}^m 1/{p_i}$. We first claim that under the assumptions of Theorem \ref{mainthm:1} (or Theorem \ref{mainthm:2}), the following result holds for any ball $\mathcal{B}$ in $\mathbb R^n$:
\begin{equation}\label{wanghua1}
\prod_{i=1}^m\bigg(\int_{\mathcal{B}} w_i(x)\,dx\bigg)^{p/{p_i}}\lesssim \int_{\mathcal{B}}\nu_{\vec{w}}(x)\,dx,
\end{equation}
provided that $w_1,\ldots,w_m\in A_\infty$ and $\nu_{\vec{w}}=\prod_{i=1}^m w_i^{p/{p_i}}$.
In fact, since $w_1,\ldots,w_m\in A_\infty$, by using Lemma \ref{Ainfty}, then we have
\begin{equation*}
\begin{split}
\prod_{i=1}^m\bigg(\int_{\mathcal{B}}w_i(x)\,dx\bigg)^{p/{p_i}}
&\leq C\prod_{i=1}^m\bigg[|\mathcal{B}|\cdot\exp\bigg(\frac{1}{|\mathcal{B}|}\int_{\mathcal{B}}\log w_i(x)\,dx\bigg)\bigg]^{p/{p_i}}\\
&=C\prod_{i=1}^m\bigg[|\mathcal{B}|^{p/{p_i}}\cdot\exp\bigg(\frac{1}{|\mathcal{B}|}\int_{\mathcal{B}}\log w_i(x)^{p/{p_i}}\,dx\bigg)\bigg]\\
&= C\cdot\big(|\mathcal{B}|\big)^{\sum_{i=1}^m p/{p_i}}\cdot
\exp\bigg(\sum_{i=1}^m\frac{1}{|\mathcal{B}|}\int_{\mathcal{B}}\log w_i(x)^{p/{p_i}}\,dx\bigg).
\end{split}
\end{equation*}
Note that
\begin{equation*}
\sum_{i=1}^m p/{p_i}=1 \quad \mathrm{and} \quad \nu_{\vec{w}}(x)=\prod_{i=1}^m w_i(x)^{p/{p_i}}.
\end{equation*}
Thus, by Jensen's inequality, we obtain
\begin{equation*}
\begin{split}
\prod_{i=1}^m\bigg(\int_{\mathcal{B}}w_i(x)\,dx\bigg)^{p/{p_i}}
&\leq C\cdot|\mathcal{B}|\cdot\exp\bigg(\frac{1}{|\mathcal{B}|}\int_{\mathcal{B}}\log\nu_{\vec{w}}(x)\,dx\bigg)\\
&\leq C\int_{\mathcal{B}}\nu_{\vec{w}}(x)\,dx.
\end{split}
\end{equation*}
This gives \eqref{wanghua1}. Moreover, in view of Lemma \ref{multi}, we have that $\nu_{\vec{w}}\in A_{mp}$ with $1/m<p<+\infty$.
This fact, together with \eqref{wanghua1} and \eqref{weights}, implies that
\begin{equation}\label{I00}
\begin{split}
I^{0,\dots,0}&\leq C\prod_{i=1}^m\big\|f_i\big\|_{L^{p_i,\kappa}(w_i)}\cdot
\frac{\nu_{\vec{w}}(2B)^{\kappa/p}}{\nu_{\vec{w}}(B)^{\kappa/p}}\\
&\leq C\prod_{i=1}^m\big\|f_i\big\|_{L^{p_i,\kappa}(w_i)}.
\end{split}
\end{equation}
To estimate the remaining terms in \eqref{I}, let us first consider the case when $\beta_1=\cdots=\beta_m=\infty$. By a simple geometric observation, we know that
\begin{equation*}
\overbrace{\big(\mathbb R^n\backslash 2B\big)\times\cdots\times\big(\mathbb R^n\backslash 2B\big)}^m
\subset(\mathbb R^n)^m\backslash (2B)^m,
\end{equation*}
and
\begin{equation*}
(\mathbb R^n)^m\backslash (2B)^m=\bigcup_{j=1}^\infty (2^{j+1}B)^m\backslash(2^{j}B)^m,
\end{equation*}
where we have used the notation $E^m=\overbrace{E\times\cdots\times E}^m$ for a measurable set $E$ and a positive integer $m$.
By the size condition \eqref{size} of the $\theta$-type Calder\'on--Zygmund kernel $K$, for any $x\in B$, we obtain
\begin{align}\label{4.1}
&\big|T_{\theta}(f^\infty_1,\ldots,f^\infty_m)(x)\big|\notag\\
&\lesssim\int_{(\mathbb R^n)^m\backslash(2B)^m}
\frac{|f_1(y_1)\cdots f_m(y_m)|}{(|x-y_1|+\cdots+|x-y_m|)^{mn}}\,dy_1\cdots dy_m\notag\\
&=\sum_{j=1}^\infty\int_{(2^{j+1}B)^m\backslash (2^{j}B)^m}
\frac{|f_1(y_1)\cdots f_m(y_m)|}{(|x-y_1|+\cdots+|x-y_m|)^{mn}}\,dy_1\cdots dy_m\notag\\
&\lesssim\sum_{j=1}^\infty\bigg(\frac{1}{|2^{j+1}B|^m}
\int_{(2^{j+1}B)^m\backslash (2^{j}B)^m}\big|f_1(y_1)\cdots f_m(y_m)\big|\,dy_1\cdots dy_m\bigg)\notag\\
&\leq \sum_{j=1}^\infty\bigg(\frac{1}{|2^{j+1}B|^m}
\prod_{i=1}^m\int_{2^{j+1}B}\big|f_i(y_i)\big|\,dy_i\bigg)\notag\\
&=\sum_{j=1}^\infty\bigg(\prod_{i=1}^m\frac{1}{|2^{j+1}B|}
\int_{2^{j+1}B}\big|f_i(y_i)\big|\,dy_i\bigg),
\end{align}
where we have used the fact that $|x-y_1|+\cdots+|x-y_m|\approx 2^{j+1}r\approx|2^{j+1}B|^{1/n}$ when $x\in B$ and $(y_1,\dots,y_m)\in (2^{j+1}B)^m\backslash (2^{j}B)^m$. Furthermore, by using H\"older's inequality, the multiple $A_{\vec{P}}$ condition on $\vec{w}$, we can deduce that
\begin{equation*}
\begin{split}
&\big|T_{\theta}(f^\infty_1,\ldots,f^\infty_m)(x)\big|\\
\lesssim&\sum_{j=1}^\infty\bigg\{\prod_{i=1}^m
\frac{1}{|2^{j+1}B|}\bigg(\int_{2^{j+1}B}\big|f_i(y_i)\big|^{p_i}w_i(y_i)\,dy_i\bigg)^{1/{p_i}}
\bigg(\int_{2^{j+1}B}w_i(y_i)^{-p'_i/{p_i}}\,dy_i\bigg)^{1/{p'_i}}\bigg\}\\
\leq&\sum_{j=1}^\infty\bigg\{\frac{1}{|2^{j+1}B|^m}\cdot
\frac{|2^{j+1}B|^{1/p+\sum_{i=1}^m(1-1/{p_i})}}{\nu_{\vec{w}}(2^{j+1}B)^{1/p}}
\prod_{i=1}^m\bigg(\big\|f_i\big\|_{L^{p_i,\kappa}(w_i)}w_i(2^{j+1}B)^{\kappa/{p_i}}\bigg)\bigg\}\\
=&\prod_{i=1}^m\big\|f_i\big\|_{L^{p_i,\kappa}(w_i)}
\sum_{j=1}^\infty\bigg\{\frac{1}{\nu_{\vec{w}}(2^{j+1}B)^{1/p}}\cdot
\prod_{i=1}^m w_i(2^{j+1}B)^{\kappa/{p_i}}\bigg\},
\end{split}
\end{equation*}
where in the last step we have used the fact that $1/p+\sum_{i=1}^m(1-1/{p_i})=m$. Hence, from the above pointwise estimate and \eqref{wanghua1}, we obtain
\begin{equation*}
\begin{split}
I^{\infty,\dots,\infty}&\lesssim\frac{\nu_{\vec{w}}(B)^{1/p}}{\nu_{\vec{w}}(B)^{\kappa/p}}
\cdot\prod_{i=1}^m\big\|f_i\big\|_{L^{p_i,\kappa}(w_i)}\sum_{j=1}^\infty\frac{\nu_{\vec{w}}(2^{j+1}B)^{\kappa/p}}{\nu_{\vec{w}}(2^{j+1}B)^{1/p}}\\
&=\prod_{i=1}^m\big\|f_i\big\|_{L^{p_i,\kappa}(w_i)}
\sum_{j=1}^\infty\frac{\nu_{\vec{w}}(B)^{{(1-\kappa)}/p}}{\nu_{\vec{w}}(2^{j+1}B)^{{(1-\kappa)}/p}}.
\end{split}
\end{equation*}
Since $\nu_{\vec{w}}\in A_{mp}\subset A_\infty$ by Lemma \ref{multi}, then it follows directly from the inequality \eqref{compare} with exponent $\delta>0$ that
\begin{equation}\label{psi1}
\frac{\nu_{\vec{w}}(B)}{\nu_{\vec{w}}(2^{j+1}B)}\lesssim\bigg(\frac{|B|}{|2^{j+1}B|}\bigg)^{\delta},
\end{equation}
which further implies
\begin{equation}\label{Iinftyty}
\begin{split}
I^{\infty,\dots,\infty}&\lesssim\prod_{i=1}^m\big\|f_i\big\|_{L^{p_i,\kappa}(w_i)}
\sum_{j=1}^\infty\bigg(\frac{|B|}{|2^{j+1}B|}\bigg)^{{\delta(1-\kappa)}/p}\\
&\lesssim\prod_{i=1}^m\big\|f_i\big\|_{L^{p_i,\kappa}(w_i)},
\end{split}
\end{equation}
where in the last estimate we have used the fact that $0<\kappa<1$ and $\delta>0$. We now consider the case where exactly $\ell$ of the $\beta_i$ are $\infty$ for some $1\le\ell<m$. We only give the arguments for one of these cases. The rest are similar and can be easily obtained from the arguments below by permuting the indices. In this case, by the same reason as above, we also have
\begin{equation*}
\overbrace{\big(\mathbb R^n\backslash 2B\big)\times\cdots\times\big(\mathbb R^n\backslash 2B\big)}^{\ell}
\subset(\mathbb R^n)^{\ell}\backslash (2B)^{\ell},
\end{equation*}
and
\begin{equation*}
(\mathbb R^n)^{\ell}\backslash (2B)^{\ell}=\bigcup_{j=1}^\infty (2^{j+1}B)^{\ell}\backslash(2^{j}B)^{\ell},\quad 1\le\ell<m.
\end{equation*}
Using the size condition \eqref{size} again, we deduce that for any $x\in B$,
\begin{align}\label{I2yr}
&\big|T_{\theta}(f^\infty_1,\ldots,f^\infty_\ell,f^0_{\ell+1},\ldots,f^0_m)(x)\big|\notag\\
&\lesssim\int_{(\mathbb R^n)^{\ell}\backslash (2B)^{\ell}}\int_{(2B)^{m-\ell}}
\frac{|f_1(y_1)\cdots f_m(y_m)|}{(|x-y_1|+\cdots+|x-y_m|)^{mn}}dy_1\cdots dy_m\notag\\
&\lesssim\prod_{i=\ell+1}^m\int_{2B}\big|f_i(y_i)\big|\,dy_i
\times\sum_{j=1}^\infty\frac{1}{|2^{j+1}B|^m}\int_{(2^{j+1}B)^\ell\backslash (2^{j}B)^\ell}
\big|f_1(y_1)\cdots f_{\ell}(y_\ell)\big|\,dy_1\cdots dy_\ell\notag\\
&\leq\prod_{i=\ell+1}^m\int_{2B}\big|f_i(y_i)\big|\,dy_i
\times\sum_{j=1}^\infty\frac{1}{|2^{j+1}B|^m}\prod_{i=1}^{\ell}
\int_{2^{j+1}B}\big|f_i(y_i)\big|\,dy_i\notag\\
&\leq\sum_{j=1}^\infty\bigg(\prod_{i=1}^m\frac{1}{|2^{j+1}B|}\int_{2^{j+1}B}\big|f_i(y_i)\big|\,dy_i\bigg),
\end{align}
where in the last inequality we have used the inclusion relation $2B\subseteq 2^{j+1}B$ with $j\in \mathbb{N}$, and hence we arrive at the same expression considered in the previous case. Hence, we can now argue exactly as we did in the estimation of $I^{\infty,\dots,\infty}$ to obtain that for all $m$-tuples $(\beta_1,\dots,\beta_m)\in\mathfrak{L}$,
\begin{align}\label{Ibeta1}
I^{\beta_1,\dots,\beta_m}
&\lesssim\prod_{i=1}^m\big\|f_i\big\|_{L^{p_i,\kappa}(w_i)}
\sum_{j=1}^\infty\frac{\nu_{\vec{w}}(B)^{{(1-\kappa)}/p}}{\nu_{\vec{w}}(2^{j+1}B)^{{(1-\kappa)}/p}}\notag\\
&\lesssim\prod_{i=1}^m\big\|f_i\big\|_{L^{p_i,\kappa}(w_i)}
\sum_{j=1}^\infty\bigg(\frac{|B|}{|2^{j+1}B|}\bigg)^{{\delta(1-\kappa)}/p}\notag\\
&\lesssim\prod_{i=1}^m\big\|f_i\big\|_{L^{p_i,\kappa}(w_i)}.
\end{align}
Combining these estimates \eqref{I00}, \eqref{Iinftyty} and \eqref{Ibeta1}, then \eqref{wangh1} holds and concludes the proof of the theorem.
\end{proof}

\begin{proof}[Proof of Theorem $\ref{mainthm:2}$]
Let $1\leq p_1,\dots,p_m<+\infty$, $\min\{p_1,\ldots,p_m\}=1$ and $\vec{f}=(f_1,\dots,f_m)$ be in $L^{p_1,\kappa}(w_1)\times\cdots
\times L^{p_m,\kappa}(w_m)$ with $\vec{w}=(w_1,\dots,w_m)\in A_{\vec{P}}$ and $0<\kappa<1$. For an arbitrary ball $B=B(x_0,r)\subset\mathbb R^n$ with $x_0\in\mathbb R^n$ and $r>0$, we need to show that the following estimate holds.
\begin{equation}\label{wangh2}
\frac{1}{\nu_{\vec{w}}(B)^{\kappa/p}}
\lambda\cdot\nu_{\vec{w}}\big(\big\{x\in B:\big|T_{\theta}(f_1,\dots,f_m)\big|>\lambda\big\}\big)^{1/p}\lesssim
\prod_{i=1}^m\big\|f_i\big\|_{L^{p_i,\kappa}(w_i)}.
\end{equation}
To this end, we represent $f_i$ as
\begin{equation*}
f_i=f_i\cdot\chi_{2B}+f_i\cdot\chi_{(2B)^{\complement}}:=f^0_i+f^{\infty}_i,\quad \mbox{for}~~i=1,2,\dots,m.
\end{equation*}
By using Lemma \ref{WMin}($N=2^m$), one can write
\begin{align}\label{Iprime}
&\frac{1}{\nu_{\vec{w}}(B)^{\kappa/p}}
\lambda\cdot\nu_{\vec{w}}\big(\big\{x\in B:\big|T_{\theta}(f_1,\dots,f_m)\big|>\lambda\big\}\big)^{1/p}\notag\\
&\leq \frac{C}{\nu_{\vec{w}}(B)^{\kappa/p}}\lambda
\cdot\nu_{\vec{w}}\big(\big\{x\in B:\big|T_{\theta}(f^0_1,\dots,f^0_m)\big|>\lambda/{2^m}\big\}\big)^{1/p}\notag\\
&+\sum_{(\beta_1,\dots,\beta_m)\in\mathfrak{L}}\frac{C}{\nu_{\vec{w}}(B)^{\kappa/p}}\lambda
\cdot\nu_{\vec{w}}\big(\big\{x\in B:\big|T_{\theta}(f^{\beta_1}_1,\dots,f^{\beta_m}_m)\big|>\lambda/{2^m}\big\}\big)^{1/p}\notag\\
&:=I^{0,\dots,0}_\ast+\sum_{(\beta_1,\dots,\beta_m)\in\mathfrak{L}} I^{\beta_1,\dots,\beta_m}_\ast,
\end{align}
where
\begin{equation*}
\mathfrak{L}=\big\{(\beta_1,\dots,\beta_m):\beta_k\in\{0,\infty\},\mbox{there is at least one $\beta_k\neq0$},1\leq k\leq m\big\}.
\end{equation*}
By the weighted weak-type estimate of $T_{\theta}$ (see Theorem \ref{weak}), we can estimate the first term on the right hand side of \eqref{Iprime} as follows.
\begin{align}\label{II0}
I^{0,\dots,0}_\ast&\leq C\cdot\frac{1}{\nu_{\vec{w}}(B)^{\kappa/p}}
\prod_{i=1}^m\bigg(\int_{2B}|f_i(x)|^{p_i}w_i(x)\,dx\bigg)^{1/{p_i}}\notag\\
&\leq C\prod_{i=1}^m\big\|f_i\big\|_{L^{p_i,\kappa}(w_i)}\frac{1}{\nu_{\vec{w}}(B)^{\kappa/p}}\cdot
\prod_{i=1}^m w_i(2B)^{\kappa/{p_i}}.
\end{align}
Moreover, in view of Lemma \ref{multi} again, we also have $\nu_{\vec{w}}\in A_{mp}$ with $1/m\leq p<+\infty$. Then we apply the inequalities \eqref{weights} and \eqref{wanghua1} to obtain that
\begin{equation}\label{WI1yr}
\begin{split}
I^{0,\dots,0}_\ast
&\leq C\prod_{i=1}^m\big\|f_i\big\|_{L^{p_i,\kappa}(w_i)}\frac{\nu_{\vec{w}}(2B)^{\kappa/p}}{\nu_{\vec{w}}(B)^{\kappa/p}}\\
&\leq C\prod_{i=1}^m\big\|f_i\big\|_{L^{p_i,\kappa}(w_i)}.
\end{split}
\end{equation}
In the proof of Theorem \ref{mainthm:1}, we have already showed the following pointwise estimate for all $m$-tuples $(\beta_1,\dots,\beta_m)\in\mathfrak{L}$ (see \eqref{4.1} and \eqref{I2yr}).
\begin{equation}
\begin{split}
&\big|T_{\theta}(f^{\beta_1}_1,\ldots,f^{\beta_m}_m)(x)\big|
\lesssim\sum_{j=1}^\infty\bigg(\prod_{i=1}^m\frac{1}{|2^{j+1}B|}\int_{2^{j+1}B}\big|f_i(y_i)\big|\,dy_i\bigg).
\end{split}
\end{equation}
Without loss of generality, we may assume that
\begin{equation*}
p_1=\cdots=p_{\ell}=\min\{p_1,\ldots,p_m\}=1 \quad \mathrm{and} \quad p_{\ell+1},\ldots,p_m>1
\end{equation*}
with $1\leq\ell<m$. The case that $p_1=\cdots=p_m=1$ can be dealt with quite similarly and more easily. Using H\"older's inequality, the multiple $A_{\vec{P}}$ condition on $\vec{w}$, we obtain that for any $x\in B$,
\begin{equation*}
\begin{split}
&\big|T_{\theta}(f^{\beta_1}_1,\ldots,f^{\beta_m}_m)(x)\big|\\
&\lesssim\sum_{j=1}^\infty\bigg(\prod_{i=1}^{\ell}\frac{1}{|2^{j+1}B|}
\int_{2^{j+1}B}\big|f_i(y_i)\big|\,dy_i\bigg)
\times\bigg(\prod_{i=\ell+1}^{m}\frac{1}{|2^{j+1}B|}\int_{2^{j+1}B}\big|f_i(y_i)\big|\,dy_i\bigg)\\
&\lesssim\sum_{j=1}^\infty\prod_{i=1}^{\ell}\frac{1}{|2^{j+1}B|}
\bigg(\int_{2^{j+1}B}\big|f_i(y_i)\big|w_i(y_i)\,dy_i\bigg)
\bigg(\inf_{y_i\in 2^{j+1}B}w_i(y_i)\bigg)^{-1}\\
&\times\prod_{i=\ell+1}^{m}\frac{1}{|2^{j+1}B|}
\bigg(\int_{2^{j+1}B}\big|f_i(y_i)\big|^{p_i}w_i(y_i)\,dy_i\bigg)^{1/{p_i}}
\bigg(\int_{2^{j+1}B}w_i(y_i)^{-p'_i/{p_i}}\,dy_i\bigg)^{1/{p'_i}}\\
&\lesssim\prod_{i=1}^m \big\|f_i\big\|_{L^{p_i,\kappa}(w_i)}
\sum_{j=1}^\infty\bigg\{\frac{1}{\nu_{\vec{w}}(2^{j+1}B)^{1/p}}\cdot
\prod_{i=1}^m w_i(2^{j+1}B)^{\kappa/{p_i}}\bigg\}\\
&\lesssim\prod_{i=1}^m \big\|f_i\big\|_{L^{p_i,\kappa}(w_i)}\sum_{j=1}^\infty\frac{1}{\nu_{\vec{w}}(2^{j+1}B)^{{(1-\kappa)}/p}},
\end{split}
\end{equation*}
where in the last inequality we have invoked \eqref{wanghua1}. Observe that $\nu_{\vec{w}}\in A_{mp}$ with $1\le mp<\infty$. Thus, it follows directly from Chebyshev's inequality and the above pointwise estimate that
\begin{align*}
I^{\beta_1,\dots,\beta_m}_\ast
&\leq \cdot\frac{C}{\nu_{\vec{w}}(B)^{\kappa/p}}
\bigg(\int_{B}\big|T_\theta(f^{\beta_1}_1,\ldots,f^{\beta_m}_m)(x)\big|^p\nu_{\vec{w}}(x)\,dx\bigg)^{1/p}\notag\\
&\leq C\prod_{i=1}^m \big\|f_i\big\|_{L^{p_i,\kappa}(w_i)}
\sum_{j=1}^\infty\frac{\nu_{\vec{w}}(B)^{{(1-\kappa)}/p}}{\nu_{\vec{w}}(2^{j+1}B)^{{(1-\kappa)}/p}}.
\end{align*}
Moreover, in view of \eqref{psi1}, we obtain that for all $m$-tuples $(\beta_1,\dots,\beta_m)\in\mathfrak{L}$,
\begin{equation}\label{WI2yr}
\begin{split}
I^{\beta_1,\dots,\beta_m}_\ast&\lesssim\prod_{i=1}^m \big\|f_i\big\|_{L^{p_i,\kappa}(w_i)}
\sum_{j=1}^\infty\bigg(\frac{|B|}{|2^{j+1}B|}\bigg)^{{\delta(1-\kappa)}/p}\\
&\lesssim\prod_{i=1}^m\big\|f_i\big\|_{L^{p_i,\kappa}(w_i)},
\end{split}
\end{equation}
where in the last step we have used the fact $\delta>0$ and $0<\kappa<1$. Putting the estimates \eqref{WI1yr} and \eqref{WI2yr} together produces the required inequality \eqref{wangh2}. Thus, by taking the supremum over all $\lambda>0$, we finish the proof of Theorem \ref{mainthm:2}.
\end{proof}
Let $1\leq p_1,\dots,p_m\leq+\infty$. We say that $\vec{w}=(w_1,\ldots,w_m)\in \prod_{i=1}^m A_{p_i}$, if each $w_i$ is in $A_{p_i}$, $i=1,2,\dots,m$. By using H\"older's inequality, it is not difficult to check that
\begin{equation*}
\prod_{i=1}^m A_{p_i}\subset A_{\vec{P}}.
\end{equation*}
Moreover, it was shown in \cite[Remark 7.2]{lerner} that this inclusion is strict. It is clear that $\prod_{i=1}^m A_{p_i}\subset \prod_{i=1}^m A_\infty$. So we have
\begin{equation}\label{include}
\prod_{i=1}^m A_{p_i}\subset A_{\vec{P}}\bigcap\prod_{i=1}^m A_\infty.
\end{equation}
A natural question appearing here is whether the above inclusion relation is also strict. Thus, as a direct consequence of Theorems \ref{mainthm:1} and \ref{mainthm:2}, we immediately obtain the following results.
\begin{corollary}
Let $m\geq 2$ and $T_{\theta}$ be an $m$-linear $\theta$-type Calder\'on--Zygmund operator with $\theta$ satisfying the condition \eqref{theta1}. If $1<p_1,\dots,p_m<+\infty$ and $1/m<p<+\infty$ with $1/p=\sum_{i=1}^m 1/{p_i}$, and $\vec{w}=(w_1,\ldots,w_m)\in\prod_{i=1}^m A_{p_i}$, then for any $0<\kappa<1$, the multilinear operator $T_{\theta}$ is bounded from $L^{p_1,\kappa}(w_1)\times L^{p_2,\kappa}(w_2)\times\cdots
\times L^{p_m,\kappa}(w_m)$ into $L^{p,\kappa}(\nu_{\vec{w}})$ with $\nu_{\vec{w}}=\prod_{i=1}^m w_i^{p/{p_i}}$.
\end{corollary}

\begin{corollary}
Let $m\geq 2$ and $T_{\theta}$ be an $m$-linear $\theta$-type Calder\'on--Zygmund operator with $\theta$ satisfying the condition \eqref{theta1}. If $1\leq p_1,\dots,p_m<+\infty$, $\min\{p_1,\ldots,p_m\}=1$ and $1/m\leq p<+\infty$ with $1/p=\sum_{i=1}^m 1/{p_i}$, and $\vec{w}=(w_1,\ldots,w_m)\in \prod_{i=1}^m A_{p_i}$, then for any $0<\kappa<1$, the multilinear operator $T_{\theta}$ is bounded from $L^{p_1,\kappa}(w_1)\times L^{p_2,\kappa}(w_2)\times\cdots\times L^{p_m,\kappa}(w_m)$ into $WL^{p,\kappa}(\nu_{\vec{w}})$ with $\nu_{\vec{w}}=\prod_{i=1}^m w_i^{p/{p_i}}$.
\end{corollary}

\section{Proofs of Theorems \ref{mainthm:3} and \ref{mainthm:4}}
To prove our main theorems for multilinear commutators in this section, we need the following lemmas about $\mathrm{BMO}$ functions.
\begin{lemma}\label{BMO}
Let $b$ be a function in $\mathrm{BMO}(\mathbb R^n)$. Then
\begin{enumerate}
  \item For every ball $B$ in $\mathbb R^n$ and for all $j\in\mathbb N$,
\begin{equation*}
\big|b_{2^{j+1}B}-b_B\big|\leq C\cdot(j+1)\|b\|_*.
\end{equation*}
  \item Let $1\leq p<+\infty$. For every ball $B$ in $\mathbb R^n$ and for all $\omega\in A_{\infty}$,
\begin{equation*}
\bigg(\int_B\big|b(x)-b_B\big|^p\omega(x)\,dx\bigg)^{1/p}\leq C\|b\|_*\cdot \omega(B)^{1/p}.
\end{equation*}
\end{enumerate}
\end{lemma}
\begin{proof}
For the proofs of the above results, we refer the reader to \cite{wang1}.
\end{proof}
Based on Lemma \ref{BMO}, we now assert that for any $j\in \mathbb{N}$ and $\omega\in A_{\infty}$, the estimate
\begin{equation}\label{j1cb}
\bigg(\int_{2^{j+1}B}\big|b(x)-b_B\big|^p\omega(x)\,dx\bigg)^{1/p}\leq C(j+1)\|b\|_*\cdot \omega(2^{j+1}B)^{1/p}
\end{equation}
holds whenever $b\in \mathrm{BMO}(\mathbb R^n)$ and $1\leq p<+\infty$. Indeed, by using Lemma \ref{BMO} (1) and (2), we could easily obtain
\begin{equation*}
\begin{split}
&\bigg(\int_{2^{j+1}B}\big|b(x)-b_B\big|^p\omega(x)\,dx\bigg)^{1/p}\\
&\leq \bigg(\int_{2^{j+1}B}\big|b(x)-b_{2^{j+1}B}\big|^p\omega(x)\,dx\bigg)^{1/p}
+\bigg(\int_{2^{j+1}B}\big|b_{2^{j+1}B}-b_B\big|^p\omega(x)\,dx\bigg)^{1/p}\\
&\leq C\|b\|_*\cdot \omega(2^{j+1}B)^{1/p}+C(j+1)\|b\|_*\cdot \omega(2^{j+1}B)^{1/p}\\
&\leq C(j+1)\|b\|_*\cdot \omega(2^{j+1}B)^{1/p},
\end{split}
\end{equation*}
as desired. Next, let us set up the following result.
\begin{lemma}\label{BMO3}
Let $b$ be a function in $\mathrm{BMO}(\mathbb R^n)$. Then for any ball $B$ in $\mathbb R^n$ and any $\omega\in A_{\infty}$, we have
\begin{equation}\label{BMOwang}
\big\|b-b_{B}\big\|_{\exp L(\omega),B}\leq C\|b\|_*.
\end{equation}
\end{lemma}
\begin{proof}
By the well-known John--Nirenberg's inequality (see \cite{john}), we know that there exist two positive constants $C_1$ and $C_2$, depending only on the dimension $n$, such that for any $\lambda>0$,
\begin{equation*}
\big|\big\{x\in B:|b(x)-b_B|>\lambda\big\}\big|\leq C_1|B|\exp\bigg\{-\frac{C_2\lambda}{\|b\|_{*}}\bigg\}.
\end{equation*}
This result shows that in some sense logarithmic growth is the maximum possible for BMO functions (more precisely, we can take $C_1=\sqrt{2}$, $C_2=\log 2/{2^{n+2}}$, see \cite[p.123--125]{duoand}). Applying the comparison property \eqref{compare} of $A_{\infty}$ weights, there is a positive number $\delta>0$ such that
\begin{equation*}
\omega\big(\big\{x\in B:|b(x)-b_B|>\lambda\big\}\big)\leq C_1\omega(B)\exp\bigg\{-\frac{C_2\delta\lambda}{\|b\|_{*}}\bigg\}.
\end{equation*}
From this, it follows that ($c_0$ and $C$ are two constants)
\begin{equation*}
\frac{1}{\omega(B)}\int_B\exp\bigg(\frac{|b(y)-b_B|}{c_0\|b\|_*}\bigg)\omega(y)\,dy\leq C,
\end{equation*}
which yields \eqref{BMOwang}.
\end{proof}
Furthermore, by \eqref{BMOwang} and Lemma \ref{BMO}(1), it is easy to check that for each $\omega$ in $A_{\infty}$ and for any ball $B$ in $\mathbb R^n$,
\begin{equation}\label{Jensen}
\big\|b-b_{B}\big\|_{\exp L(\omega),2^{j+1}B}\leq C(j+1)\|b\|_\ast,\quad j\in \mathbb{N}.
\end{equation}
We are now in a position to give the proofs of Theorems \ref{mainthm:3} and \ref{mainthm:4}.
\begin{proof}[Proof of Theorem $\ref{mainthm:3}$]
Let $1<p_1,\dots,p_m<+\infty$ and $\vec{f}=(f_1,\dots,f_m)$ be in $L^{p_1,\kappa}(w_1)\times\cdots
\times L^{p_m,\kappa}(w_m)$ with $\vec{w}=(w_1,\dots,w_m)\in A_{\vec{P}}$ and $0<\kappa<1$. As was pointed out in \cite{lerner}, by linearity it is enough to consider the multilinear commutator $[\Sigma b,T_{\theta}]$ with only one symbol. Without loss of generality, we fix $b\in \mathrm{BMO}(\mathbb R^n)$, and then consider the operator
\begin{equation*}
\big[b,T_\theta\big]_1(\vec{f})(x)=b(x)\cdot T_\theta(f_1,f_2,\dots,f_m)(x)-T_{\theta}(bf_1,f_2,\dots,f_m)(x).
\end{equation*}
For each fixed ball $B=B(x_0,r)\subset\mathbb R^n$, it is enough to prove that
\begin{equation}\label{wangh3}
\frac{1}{\nu_{\vec{w}}(B)^{\kappa/p}}
\bigg(\int_{B}\big|\big[b,T_{\theta}\big]_1(f_1,\dots,f_m)(x)\big|^p\nu_{\vec{w}}(x)\,dx\bigg)^{1/p}
\lesssim\|b\|_\ast\prod_{i=1}^m\big\|f_i\big\|_{L^{p_i,\kappa}(w_i)}.
\end{equation}
As before, we decompose $f_i$ as
$f_i=f^0_i+f^{\infty}_0$,
where $f^i_0=f_i\cdot\chi_{2B}$ and $f^{\infty}_i=f_i\cdot\chi_{(2B)^{\complement}}$, $i=1,2,\dots,m$. We set $tB=B(x_0,tr)$ for any $t>0$. Let $\mathfrak{L}$ be the same as before. By using Lemma \ref{Min}($N=2^m$), we can write
\begin{align}\label{J}
&\frac{1}{\nu_{\vec{w}}(B)^{\kappa/p}}
\bigg(\int_{B}\big|\big[b,T_{\theta}\big]_1(f_1,\dots,f_m)(x)\big|^p\nu_{\vec{w}}(x)\,dx\bigg)^{1/p}\notag\\
&\leq C\cdot\frac{1}{\nu_{\vec{w}}(B)^{\kappa/p}}
\bigg(\int_{B}\big|\big[b,T_{\theta}\big]_1(f^0_1,\dots,f^0_m)(x)\big|^p\nu_{\vec{w}}(x)\,dx\bigg)^{1/p}\notag\\
&+C\sum_{(\beta_1,\dots,\beta_m)\in\mathfrak{L}}\frac{1}{\nu_{\vec{w}}(B)^{\kappa/p}}
\bigg(\int_{B}\big|\big[b,T_{\theta}\big]_1(f^{\beta_1}_1,\ldots,f^{\beta_m}_m)(x)\big|^p\nu_{\vec{w}}(x)\,dx\bigg)^{1/p}\notag\\
&:=J^{0,\dots,0}+\sum_{(\beta_1,\dots,\beta_m)\in\mathfrak{L}} J^{\beta_1,\dots,\beta_m}.
\end{align}
To estimate the first summand of \eqref{J}, applying Theorem \ref{comm} along with \eqref{weights} and \eqref{wanghua1}, we get
\begin{align}\label{J1}
J^{0,\dots,0}&\leq C\cdot\frac{1}{\nu_{\vec{w}}(B)^{\kappa/p}}
\prod_{i=1}^m\bigg(\int_{2B}|f_i(x)|^{p_i}w_i(x)\,dx\bigg)^{1/{p_i}}\notag\\
&\leq C\prod_{i=1}^m\big\|f_i\big\|_{L^{p_i,\kappa}(w_i)}\cdot
\frac{1}{\nu_{\vec{w}}(B)^{\kappa/p}}\prod_{i=1}^m w_i(2B)^{\kappa/{p_i}}\notag\\
&\leq C\prod_{i=1}^m\big\|f_i\big\|_{L^{p_i,\kappa}(w_i)}\cdot
\frac{\nu_{\vec{w}}(2B)^{\kappa/p}}{\nu_{\vec{w}}(B)^{\kappa/p}}\notag\\
&\leq C\prod_{i=1}^m\big\|f_i\big\|_{L^{p_i,\kappa}(w_i)}.
\end{align}
To estimate the remaining terms in \eqref{J}, let us first consider the case when $\beta_1=\cdots=\beta_m=\infty$. It is easy to see that for any $x\in B$,
\begin{equation*}
\big[b,T_\theta\big]_1(\vec{f})(x)=[b(x)-b_B]\cdot T_\theta(f_1,f_2,\dots,f_m)(x)-T_{\theta}((b-b_B)f_1,f_2,\dots,f_m)(x).
\end{equation*}
Hence, we divide the term $J^{\infty,\dots,\infty}$ into two parts below.
\begin{equation*}
\begin{split}
J^{\infty,\dots,\infty}&\leq C\cdot\frac{1}{\nu_{\vec{w}}(B)^{\kappa/p}}
\bigg(\int_{B}\big|[b(x)-b_B]\cdot T_\theta(f^\infty_1,f^{\infty}_2,\dots,f^\infty_m)(x)\big|^p\nu_{\vec{w}}(x)\,dx\bigg)^{1/p}\\
&+C\cdot\frac{1}{\nu_{\vec{w}}(B)^{\kappa/p}}
\bigg(\int_{B}\big|T_\theta((b-b_B)f^\infty_1,f^{\infty}_2,\dots,f^\infty_m)(x)\big|^p
\nu_{\vec{w}}(x)\,dx\bigg)^{1/p}\\
&:=J^{\infty,\dots,\infty}_{\star}+J^{\infty,\dots,\infty}_{\star\star}.
\end{split}
\end{equation*}
Next, we estimate each term separately. In the proof of Theorem \ref{mainthm:1}, we have already shown that (see \eqref{4.1})
\begin{equation*}
\big|T_{\theta}(f^\infty_1,f^{\infty}_2,\ldots,f^\infty_m)(x)\big|
\lesssim\sum_{j=1}^\infty\bigg(\prod_{i=1}^m\frac{1}{|2^{j+1}B|}\int_{2^{j+1}B}\big|f_i(y_i)\big|\,dy_i\bigg).
\end{equation*}
Note that $\nu_{\vec{w}}\in A_{mp}\subset A_{\infty}$. From Lemma \ref{BMO}(2), it follows that
\begin{equation*}
\begin{split}
J^{\infty,\dots,\infty}_{\star}
&\lesssim\frac{1}{\nu_{\vec{w}}(B)^{\kappa/p}}
\sum_{j=1}^\infty\bigg(\prod_{i=1}^m\frac{1}{|2^{j+1}B|}
\int_{2^{j+1}B}\big|f_i(y_i)\big|\,dy_i\bigg)\\
&\times\bigg(\int_{B}\big|b(x)-b_B\big|^p\nu_{\vec{w}}(x)\,dx\bigg)^{1/p}\\
&\lesssim\|b\|_\ast\cdot\nu_{\vec{w}}(B)^{1/{p}-\kappa/p}
\sum_{j=1}^\infty\bigg(\prod_{i=1}^m\frac{1}{|2^{j+1}B|}
\int_{2^{j+1}B}\big|f_i(y_i)\big|\,dy_i\bigg).
\end{split}
\end{equation*}
We then follow the same arguments as in the proof of Theorem \ref{mainthm:1} to get
\begin{align}\label{5.7}
J^{\infty,\dots,\infty}_{\star}
&\lesssim\|b\|_\ast\prod_{i=1}^m \big\|f_i\big\|_{L^{p_i,\kappa}(w_i)}
\sum_{j=1}^\infty\frac{\nu_{\vec{w}}(B)^{{(1-\kappa)}/p}}{\nu_{\vec{w}}(2^{j+1}B)^{{(1-\kappa)}/p}}\notag\\
&\lesssim\|b\|_\ast\prod_{i=1}^m \big\|f_i\big\|_{L^{p_i,\kappa}(w_i)}.
\end{align}
Using the same methods as in Theorem \ref{mainthm:1}, we can also deduce that
\begin{equation}\label{pointwise2}
\begin{split}
&\big|T_\theta((b-b_B)f^\infty_1,f^\infty_2,\dots,f^\infty_m)(x)\big|\\
&\lesssim\int_{(\mathbb R^n)^m\backslash (2B)^m}
\frac{|(b(y_1)-b_{B})f_1(y_1)|\cdot|f_2(y_2)\cdots f_m(y_m)|}{(|x-y_1|+\cdots+|x-y_m|)^{mn}}\,dy_1\cdots dy_m\notag\\
&=\sum_{j=1}^\infty\int_{(2^{j+1}B)^m\backslash (2^{j}B)^m}
\frac{|(b(y_1)-b_{B})f_1(y_1)|\cdot|f_2(y_2)\cdots f_m(y_m)|}{(|x-y_1|+\cdots+|x-y_m|)^{mn}}\,dy_1\cdots dy_m\notag\\
&\lesssim\sum_{j=1}^\infty\bigg(\frac{1}{|2^{j+1}B|^m}
\int_{(2^{j+1}B)^m\backslash (2^{j}B)^m}|(b(y_1)-b_{B})f_1(y_1)|\cdot\big|f_2(y_2)\cdots f_m(y_m)\big|\,dy_1\cdots dy_m\bigg)\notag\\
&\leq \sum_{j=1}^\infty\bigg(\frac{1}{|2^{j+1}B|^m}\int_{2^{j+1}B}|(b(y_1)-b_{B})f_1(y_1)|\,dy_1
\prod_{i=2}^m
\int_{2^{j+1}B}\big|f_i(y_i)\big|\,dy_i\bigg)\notag\\
&=\sum_{j=1}^\infty\bigg(\frac{1}{|2^{j+1}B|}\int_{2^{j+1}B}|(b(y_1)-b_{B})f_1(y_1)|\,dy_1\bigg)
\bigg(\prod_{i=2}^m\frac{1}{|2^{j+1}B|}\int_{2^{j+1}B}\big|f_i(y_i)\big|\,dy_i\bigg).
\end{split}
\end{equation}
Then we have
\begin{equation}\label{substi}
\begin{split}
J^{\infty,\dots,\infty}_{\star\star}&\lesssim\nu_{\vec{w}}(B)^{{(1-\kappa)}/p}\\
&\times\sum_{j=1}^\infty\bigg(\frac{1}{|2^{j+1}B|}\int_{2^{j+1}B}|(b(y_1)-b_{B})f_1(y_1)|\,dy_1\bigg)
\bigg(\prod_{i=2}^m\frac{1}{|2^{j+1}B|}\int_{2^{j+1}B}\big|f_i(y_i)\big|\,dy_i\bigg).
\end{split}
\end{equation}
For each $2\leq i\leq m$, by using H\"older's inequality with exponent $p_i$, we obtain that
\begin{equation*}
\int_{2^{j+1}B}\big|f_i(y_i)\big|\,dy_i\leq
\bigg(\int_{2^{j+1}B}\big|f_i(y_i)\big|^{p_i}w_i(y_i)\,dy_i\bigg)^{1/{p_i}}
\bigg(\int_{2^{j+1}B}w_i(y_i)^{-p'_i/{p_i}}\,dy_i\bigg)^{1/{p'_i}}.
\end{equation*}
According to Lemma \ref{multi}, we have $w_i^{1-p'_i}=w_i^{-p'_i/{p_i}}\in A_{mp'_i}\subset A_{\infty}$, $i=1,2,\dots,m$. By using H\"older's inequality again with exponent $p_1$ and \eqref{j1cb}, we deduce that
\begin{equation*}
\begin{split}
&\int_{2^{j+1}B}|(b(y_1)-b_{B})f_1(y_1)|\,dy_1\\
&\leq\bigg(\int_{2^{j+1}B}\big|f_1(y_1)\big|^{p_1}w_1(y_1)\,dy_1\bigg)^{1/{p_1}}
\bigg(\int_{2^{j+1}B}|b(y_1)-b_{B}|^{p'_1}w_1(y_1)^{-p'_1/{p_1}}\,dy_1\bigg)^{1/{p'_1}}\\
&\lesssim\bigg(\int_{2^{j+1}B}\big|f_1(y_1)\big|^{p_1}w_1(y_1)\,dy_1\bigg)^{1/{p_1}}
(j+1)\|b\|_*\cdot\bigg(\int_{2^{j+1}B}w_1(y_1)^{-p'_1/{p_1}}\,dy_1\bigg)^{1/{p'_1}},
\end{split}
\end{equation*}
where the last inequality is valid by the fact that $w_1^{-p'_1/{p_1}}\in A_{\infty}$. Substituting the above two estimates into the formula \eqref{substi}, we have
\begin{equation*}
\begin{split}
J^{\infty,\dots,\infty}_{\star\star}&\lesssim\|b\|_*\cdot\nu_{\vec{w}}(B)^{{(1-\kappa)}/p}\\
&\sum_{j=1}^\infty(j+1)\bigg\{\prod_{i=1}^m
\frac{1}{|2^{j+1}B|}\bigg(\int_{2^{j+1}B}\big|f_i(y_i)\big|^{p_i}w_i(y_i)\,dy_i\bigg)^{1/{p_i}}
\bigg(\int_{2^{j+1}B}w_i(y_i)^{-p'_i/{p_i}}\,dy_i\bigg)^{1/{p'_i}}\bigg\}\\
&\lesssim\|b\|_*\cdot\nu_{\vec{w}}(B)^{{(1-\kappa)}/p}\sum_{j=1}^\infty(j+1)\bigg\{\frac{1}{\nu_{\vec{w}}(2^{j+1}B)^{1/p}}
\prod_{i=1}^m\bigg(\big\|f_i\big\|_{L^{p_i,\kappa}(w_i)}w_i(2^{j+1}B)^{\kappa/{p_i}}\bigg)\bigg\}\\
&\lesssim\|b\|_\ast\prod_{i=1}^m \big\|f_i\big\|_{L^{p_i,\kappa}(w_i)}
\sum_{j=1}^\infty(j+1)\cdot\frac{\nu_{\vec{w}}(B)^{{(1-\kappa)}/p}}{\nu_{\vec{w}}(2^{j+1}B)^{{(1-\kappa)}/p}},
\end{split}
\end{equation*}
where in the last two inequalities we have used the $A_{\vec{P}}$ condition and \eqref{wanghua1}. Moreover, in view of \eqref{psi1}(since $\nu_{\vec{w}}\in A_{mp}$ with $1<mp<+\infty$), the last expression is bounded by
\begin{align}\label{5.9}
&\|b\|_\ast\prod_{i=1}^m \big\|f_i\big\|_{L^{p_i,\kappa}(w_i)}
\sum_{j=1}^\infty(j+1)\cdot\left(\frac{|B|}{|2^{j+1}B|}\right)^{\delta{(1-\kappa)}/p}\notag\\
\lesssim &\|b\|_\ast\prod_{i=1}^m \big\|f_i\big\|_{L^{p_i,\kappa}(w_i)},
\end{align}
where the last series is convergent since the exponent $\delta {(1-\kappa)}/p$ is positive. Consequently, combining the inequality \eqref{5.9} with \eqref{5.7}, we get
\begin{equation*}
J^{\infty,\dots,\infty}\lesssim \|b\|_\ast\prod_{i=1}^m \big\|f_i\big\|_{L^{p_i,\kappa}(w_i)}.
\end{equation*}
We now consider the case where exactly $\ell$ of the $\beta_i$ are $\infty$ for some $1\le\ell<m$. We only give the arguments for one of these cases. The rest are similar and can be easily obtained from the arguments below by permuting the indices. Meanwhile, we consider only the case $\beta_1=\infty$ here since the other case can be proved in the same way. We now estimate the term $\big|\big[b,T_{\theta}\big]_1(f^{\beta_1}_1,\ldots,f^{\beta_m}_m)(x)\big|$ when
\begin{equation*}
\beta_1=\cdots=\beta_{\ell}=\infty\quad \&\quad\beta_{\ell+1}=\cdots=\beta_m=0.
\end{equation*}
In our present situation, we first divide the term $J^{\beta_1,\dots,\beta_m}$ into two parts as follows.
\begin{equation*}
\begin{split}
J^{\beta_1,\dots,\beta_m}&\leq C\cdot\frac{1}{\nu_{\vec{w}}(B)^{\kappa/p}}
\bigg(\int_{B}\big|[b(x)-b_B]\cdot T_\theta(f^\infty_1,\ldots,f^\infty_\ell,f^0_{\ell+1},\ldots,f^0_m)(x)\big|^p\nu_{\vec{w}}(x)\,dx\bigg)^{1/p}\\
&+C\cdot\frac{1}{\nu_{\vec{w}}(B)^{\kappa/p}}
\bigg(\int_{B}\big|T_\theta((b-b_B)f^\infty_1,\ldots,f^\infty_\ell,f^0_{\ell+1},\ldots,f^0_m)(x)\big|^p
\nu_{\vec{w}}(x)\,dx\bigg)^{1/p}\\
&:=J^{\beta_1,\dots,\beta_m}_{\star}+J^{\beta_1,\dots,\beta_m}_{\star\star}.
\end{split}
\end{equation*}
Next, we estimate each term respectively. Recall that the following result has been proved in Theorem \ref{mainthm:1}(see \eqref{I2yr}).
\begin{equation*}
\big|T_{\theta}(f^\infty_1,\ldots,f^\infty_\ell,f^0_{\ell+1},\ldots,f^0_m)(x)\big|
\lesssim\sum_{j=1}^\infty\bigg(\prod_{i=1}^m\frac{1}{|2^{j+1}B|}\int_{2^{j+1}B}\big|f_i(y_i)\big|\,dy_i\bigg).
\end{equation*}
From Lemma \ref{BMO}(2), it then follows that
\begin{equation*}
\begin{split}
J^{\beta_1,\dots,\beta_m}_{\star}
&\lesssim\frac{1}{\nu_{\vec{w}}(B)^{\kappa/p}}
\sum_{j=1}^\infty\bigg(\prod_{i=1}^m\frac{1}{|2^{j+1}B|}
\int_{2^{j+1}B}\big|f_i(y_i)\big|\,dy_i\bigg)\\
&\times\bigg(\int_{B}\big|b(x)-b_B\big|^p\nu_{\vec{w}}(x)\,dx\bigg)^{1/p}\\
&\lesssim\|b\|_\ast\cdot\nu_{\vec{w}}(B)^{1/{p}-\kappa/p}
\sum_{j=1}^\infty\bigg(\prod_{i=1}^m\frac{1}{|2^{j+1}B|}
\int_{2^{j+1}B}\big|f_i(y_i)\big|\,dy_i\bigg).
\end{split}
\end{equation*}
We now proceed exactly as we did in the proof of Theorem \ref{mainthm:1} to obtain that
\begin{align}\label{5.12}
J^{\beta_1,\dots,\beta_m}_{\star}
&\lesssim\|b\|_\ast\prod_{i=1}^m \big\|f_i\big\|_{L^{p_i,\kappa}(w_i)}
\sum_{j=1}^\infty\frac{\nu_{\vec{w}}(B)^{{(1-\kappa)}/p}}{\nu_{\vec{w}}(2^{j+1}B)^{{(1-\kappa)}/p}}\notag\\
&\lesssim\|b\|_\ast\prod_{i=1}^m \big\|f_i\big\|_{L^{p_i,\kappa}(w_i)}.
\end{align}
On the other hand, by adopting the same method given in Theorem \ref{mainthm:1}, we can see that
\begin{align}\label{jj2yr}
&\big|T_{\theta}((b-b_B)f^\infty_1,\ldots,f^\infty_\ell,f^0_{\ell+1},\ldots,f^0_m)(x)\big|\\
&\lesssim\int_{(\mathbb R^n)^{\ell}\backslash (2B)^{\ell}}\int_{(2B)^{m-\ell}}
\frac{|(b(y_1)-b_{B})f_1(y_1)|\cdot|f_2(y_2)\cdots f_m(y_m)|}{(|x-y_1|+\cdots+|x-y_m|)^{mn}}dy_1\cdots dy_m\notag\\
&\lesssim\prod_{i=\ell+1}^m\int_{2B}\big|f_i(y_i)\big|\,dy_i\notag\\
&\times\sum_{j=1}^\infty\frac{1}{|2^{j+1}B|^m}\int_{(2^{j+1}B)^\ell\backslash (2^{j}B)^\ell}
|(b(y_1)-b_{B})f_1(y_1)|\cdot\big|f_2(y_2)\cdots f_{\ell}(y_\ell)\big|\,dy_1\cdots dy_\ell\notag\\
&\leq\prod_{i=\ell+1}^m\int_{2B}\big|f_i(y_i)\big|\,dy_i\notag\\
&\times\sum_{j=1}^\infty\frac{1}{|2^{j+1}B|^m}\int_{2^{j+1}B}|(b(y_1)-b_{B})f_1(y_1)|\,dy_1\prod_{i=2}^{\ell}
\int_{2^{j+1}B}\big|f_i(y_i)\big|\,dy_i\notag\\
&\leq\sum_{j=1}^\infty\bigg(\frac{1}{|2^{j+1}B|^m}\int_{2^{j+1}B}|(b(y_1)-b_{B})f_1(y_1)|\,dy_1
\prod_{i=2}^m\int_{2^{j+1}B}\big|f_i(y_i)\big|\,dy_i\bigg)\notag,
\end{align}
where in the last inequality we have used the inclusion relation $2B\subseteq 2^{j+1}B$ with $j\in \mathbb{N}$. For the same reason as above, we get the desired estimate.
\begin{align}\label{5.14}
J^{\beta_1,\dots,\beta_m}_{\star\star}
&\lesssim\|b\|_\ast\prod_{i=1}^m \big\|f_i\big\|_{L^{p_i,\kappa}(w_i)}
\sum_{j=1}^\infty(j+1)\cdot\frac{\nu_{\vec{w}}(B)^{{(1-\kappa)}/p}}{\nu_{\vec{w}}(2^{j+1}B)^{{(1-\kappa)}/p}}\notag\\
&\lesssim\|b\|_\ast\prod_{i=1}^m \big\|f_i\big\|_{L^{p_i,\kappa}(w_i)}.
\end{align}
Combining \eqref{5.12} and \eqref{5.14}, we conclude that
\begin{equation*}
J^{\beta_1,\dots,\beta_m}\lesssim\|b\|_\ast\prod_{i=1}^m \big\|f_i\big\|_{L^{p_i,\kappa}(w_i)}.
\end{equation*}
Summarizing the estimates derived above, then \eqref{wangh3} holds and hence the proof of Theorem \ref{mainthm:3} is complete.
\end{proof}

\begin{proof}[Proof of Theorem $\ref{mainthm:4}$]
Given $\vec{f}=(f_1,f_2,\dots,f_m)$, for any fixed ball $B=B(x_0,r)$ in $\mathbb R^n$, as before, we decompose each $f_i$ as
\begin{equation*}
f_i=f^0_i+f^{\infty}_i,~~i=1,2,\dots,m,
\end{equation*}
where $f^0_i=f_i\cdot\chi_{2B}$,$f^{\infty}_i=f_i\cdot\chi_{(2B)^\complement}$ and $2B=B(x,2r)\subset\mathbb R^n$. Again, we only consider here the multilinear commutator with only one symbol by linearity; that is, fix $b\in\mathrm{BMO}(\mathbb R^n)$ and consider the operator
\begin{equation*}
\big[b,T_\theta\big]_1(\vec{f})(x)=b(x)\cdot T_\theta(f_1,f_2,\dots,f_m)(x)-T_{\theta}(bf_1,f_2,\dots,f_m)(x).
\end{equation*}
Let $\mathfrak{L}$ be the same as before. Then for any given $\lambda>0$, by using Lemma \ref{WMin}($N=2^m$), one can write
\begin{align*}
&\frac{1}{\nu_{\vec{w}}(B)^{m\kappa}}
\cdot\Big[\nu_{\vec{w}}\big(\big\{x\in B:\big|\big[{b},T_\theta\big]_1(\vec{f})(x)\big|>\lambda^m\big\}\big)\Big]^m\notag\\
&\leq \frac{C}{\nu_{\vec{w}}(B)^{m\kappa}}
\cdot\Big[\nu_{\vec{w}}\big(\big\{x\in B:\big|\big[{b},T_\theta\big]_1(f^0_1,\dots,f^0_m)(x)\big|>\lambda^m/{2^m}\big\}\big)\Big]^m\notag\\
&+\sum_{(\beta_1,\dots,\beta_m)\in\mathfrak{L}}\frac{C}{\nu_{\vec{w}}(B)^{m\kappa}}
\cdot\Big[\nu_{\vec{w}}\big(\big\{x\in B:\big|\big[{b},T_\theta\big]_1(f^{\beta_1}_1,\ldots,f^{\beta_m}_m)(x)\big|>\lambda^m/{2^m}\big\}\big)\Big]^m\notag\\
&:=J^{0,\dots,0}_\ast+\sum_{(\beta_1,\dots,\beta_m)\in\mathfrak{L}} J^{\beta_1,\dots,\beta_m}_\ast.
\end{align*}
Observe that the Young function $\Phi(t)=t\cdot(1+\log^+t)$ satisfies the doubling condition, that is, there is a constant $C_{\Phi}>0$ such that for every $t>0$,
\begin{equation*}
\Phi(2t)\leq C_{\Phi}\,\Phi(t).
\end{equation*}
This fact together with Theorem \ref{Wcomm} yields
\begin{equation*}
\begin{split}
J^{0,\dots,0}_\ast&\leq \frac{C}{\nu_{\vec{w}}(B)^{m\kappa}}
\prod_{i=1}^m\bigg(\int_{\mathbb R^n}\Phi\bigg(\frac{2|f^0_i(x)|}{\lambda}\bigg)\cdot w_i(x)\,dx\bigg)\\
&\leq \frac{C}{\nu_{\vec{w}}(B)^{m\kappa}}
\prod_{i=1}^m\bigg(\int_{2B}\Phi\bigg(\frac{|f_i(x)|}{\lambda}\bigg)\cdot w_i(x)\,dx\bigg)\\
&=\frac{C}{\nu_{\vec{w}}(B)^{m\kappa}}
\prod_{i=1}^m w_i(2B)
\bigg(\frac{1}{w_i(2B)}\int_{2B}\Phi\bigg(\frac{|f_i(x)|}{\lambda}\bigg)\cdot w_i(x)\,dx\bigg)\\
&\leq \frac{C}{\nu_{\vec{w}}(B)^{m\kappa}}
\prod_{i=1}^m w_i(2B)\cdot\bigg\|\Phi\bigg(\frac{|f_i|}{\lambda}\bigg)\bigg\|_{L\log L(w_i),2B},
\end{split}
\end{equation*}
where in the last inequality we have used the estimate \eqref{main esti1}. Since $\vec{w}=(w_1,\ldots,w_m)\in A_{(1,\dots,1)}$, by definition, we know that
\begin{equation}\label{A11}
\bigg(\frac{1}{|\mathcal{B}|}\int_{\mathcal{B}} \nu_{\vec{w}}(x)\,dx\bigg)^{m}
\leq C\prod_{i=1}^m\inf_{x\in \mathcal{B}}w_i(x)
\end{equation}
holds for any ball $\mathcal{B}$ in $\mathbb R^n$, where $\nu_{\vec{w}}=\prod_{i=1}^m w_i^{1/{m}}$. We can rewrite this inequality as
\begin{equation*}
\begin{split}
\bigg(\frac{1}{|\mathcal{B}|}\int_{\mathcal{B}}\nu_{\vec{w}}(x)\,dx\bigg)
&\leq C\bigg(\prod_{i=1}^m\inf_{x\in \mathcal{B}}w_i(x)\bigg)^{1/m}
=C\bigg(\prod_{i=1}^m\inf_{x\in \mathcal{B}}w_i(x)^{1/m}\bigg)\\
&\leq C\bigg(\inf_{x\in \mathcal{B}}\prod_{i=1}^mw_i(x)^{1/m}\bigg)=C\cdot\inf_{x\in \mathcal{B}}\nu_{\vec{w}}(x),
\end{split}
\end{equation*}
which means that $\nu_{\vec{w}}\in A_1$. Moreover, for each $w_i$, $i=1,2,\dots,m$, it is easy to see that
\begin{equation*}
\begin{split}
\bigg(\prod_{j\neq i}\inf_{x\in \mathcal{B}}w_j(x)^{1/{m}}\bigg)^m\bigg(\frac{1}{|\mathcal{B}|}\int_{\mathcal{B}}w_i(x)^{1/{m}}\,dx\bigg)^{m}
&\leq\bigg(\frac{1}{|\mathcal{B}|}\int_{\mathcal{B}}w_i(x)^{1/{m}}\cdot\prod_{j\neq i}w_j(x)^{1/{m}}\,dx\bigg)^{m}\\
&\leq C\prod_{j=1}^m\inf_{x\in \mathcal{B}}w_j(x).
\end{split}
\end{equation*}
Also observe that
\begin{equation*}
\bigg(\prod_{j\neq i}\inf_{x\in \mathcal{B}}w_j(x)^{1/{m}}\bigg)^m=\prod_{j\neq i}\inf_{x\in \mathcal{B}}w_j(x).
\end{equation*}
From this, it follows that
\begin{equation*}
\bigg(\frac{1}{|\mathcal{B}|}\int_{\mathcal{B}}w_i(x)^{1/{m}}\,dx\bigg)^{m}\leq C\cdot\inf_{x\in \mathcal{B}}w_i(x),
\end{equation*}
which implies that $w_i^{1/{m}}\in A_1$ ($i=1,2,\dots,m$). Thus, by the inequality \eqref{weights} and \eqref{wanghua1}(taking $p_1=\cdots=p_m=1$ and $p=1/m$), we have
\begin{equation*}
\begin{split}
J^{0,\dots,0}_\ast
&\lesssim\prod_{i=1}^m\bigg\|\Phi\bigg(\frac{|f_i|}{\lambda}\bigg)\bigg\|_{(L\log L)^{1,\kappa}(w_i)}
\frac{1}{\nu_{\vec{w}}(B)^{m\kappa}}\cdot
\prod_{i=1}^m w_i(2B)^{\kappa}\\
&\lesssim\prod_{i=1}^m\bigg\|\Phi\bigg(\frac{|f_i|}{\lambda}\bigg)\bigg\|_{(L\log L)^{1,\kappa}(w_i)}
\cdot\frac{\nu_{\vec{w}}(2B)^{m\kappa}}{\nu_{\vec{w}}(B)^{m\kappa}}\\
&\lesssim\prod_{i=1}^m\bigg\|\Phi\bigg(\frac{|f_i|}{\lambda}\bigg)\bigg\|_{(L\log L)^{1,\kappa}(w_i)}.
\end{split}
\end{equation*}
It remains to estimate the term $J^{\beta_1,\dots,\beta_m}_\ast$ for $(\beta_1,\dots,\beta_m)\in\mathfrak{L}$. Recall that for any $x\in B$,
\begin{equation*}
\big[b,T_\theta\big]_1(\vec{f})(x)=[b(x)-b_B]\cdot T_\theta(f_1,f_2,\dots,f_m)(x)-T_{\theta}((b-b_B)f_1,f_2,\dots,f_m)(x).
\end{equation*}
So we can further decompose $J^{\beta_1,\dots,\beta_m}_\ast$ as
\begin{equation*}
\begin{split}
J^{\beta_1,\dots,\beta_m}_\ast
\leq&\frac{C}{\nu_{\vec{w}}(B)^{m\kappa}}
\Big[\nu_{\vec{w}}\Big(\Big\{x\in B:\big|[b(x)-b_B]\cdot T_\theta(f^{\beta_1}_1,f^{\beta_2}_2,\ldots,f^{\beta_m}_m)(x)\big|>\lambda^m/{2^{m+1}}\Big\}\Big)\Big]^m\\
&+\frac{C}{\nu_{\vec{w}}(B)^{m\kappa}}
\Big[\nu_{\vec{w}}\Big(\Big\{x\in B:\big|T_{\theta}((b-b_B)f^{\beta_1}_1,f^{\beta_2}_2,\dots,f^{\beta_m}_m)(x)\big|>\lambda^m/{2^{m+1}}\Big\}\Big)\Big]^m\\
:=&\widetilde{J}^{\beta_1,\dots,\beta_m}_\star+\widetilde{J}^{\beta_1,\dots,\beta_m}_{\star\star}.
\end{split}
\end{equation*}
By using the previous pointwise estimates \eqref{4.1} and \eqref{I2yr} together with Chebyshev's inequality, we can deduce that
\begin{equation*}
\begin{split}
\widetilde{J}^{\beta_1,\dots,\beta_m}_\star&\leq \frac{C}{\nu_{\vec{w}}(B)^{m\kappa}}
\times\frac{2^{m+1}}{\lambda^m}
\bigg(\int_{B}\big|[b(x)-b_{B}]\cdot T_\theta(f^{\beta_1}_1,f^{\beta_2}_2,\ldots,f^{\beta_m}_m)(x)\big|^{\frac{\,1\,}{m}}\nu_{\vec{w}}(x)\,dx\bigg)^m\\
&\leq \frac{C}{\nu_{\vec{w}}(B)^{m\kappa}}
\sum_{j=1}^\infty\bigg(\prod_{i=1}^m\frac{1}{|2^{j+1}B|}
\int_{2^{j+1}B}\frac{|f_i(y_i)|}{\lambda}\,dy_i\bigg)\\
&\times\bigg(\int_B\big|b(x)-b_{B}\big|^{\frac{\,1\,}{m}}\nu_{\vec{w}}(x)\,dx\bigg)^m.
\end{split}
\end{equation*}
We claim that for $2\leq m\in \mathbb{N}$ and $\nu_{\vec{w}}\in A_1$,
\begin{equation}\label{assertion}
\bigg(\int_B\big|b(x)-b_{B}\big|^{\frac{\,1\,}{m}}\nu_{\vec{w}}(x)\,dx\bigg)^m\lesssim \|b\|_*\cdot\nu_{\vec{w}}(B)^{m}.
\end{equation}
Assuming the claim \eqref{assertion} holds for the moment, then we have
\begin{equation*}
\widetilde{J}^{\beta_1,\dots,\beta_m}_\star
\lesssim\|b\|_*\cdot\nu_{\vec{w}}(B)^{m(1-\kappa)}
\sum_{j=1}^\infty\bigg(\prod_{i=1}^m\frac{1}{|2^{j+1}B|}
\int_{2^{j+1}B}\frac{|f_i(y_i)|}{\lambda}\,dy_i\bigg).
\end{equation*}
Furthermore, note that $t\leq\Phi(t)=t\cdot(1+\log^+t)$ for any $t>0$. This fact along with the multiple $A_{(1,\dots,1)}$ condition \eqref{A11} implies that
\begin{equation*}
\begin{split}
\widetilde{J}^{\beta_1,\dots,\beta_m}_\star
&\lesssim\|b\|_*\cdot\nu_{\vec{w}}(B)^{m(1-\kappa)}\\
&\times\sum_{j=1}^\infty\prod_{i=1}^m\bigg(\frac{1}{|2^{j+1}B|}
\int_{2^{j+1}B}\frac{|f_i(y_i)|}{\lambda}\cdot w_i(y_i)\,dy_i\bigg)
\left(\inf_{y_i\in 2^{j+1}B}w_i(y_i)\right)^{-1}\\
&\lesssim\|b\|_*\cdot\nu_{\vec{w}}(B)^{m(1-\kappa)}
\times\sum_{j=1}^\infty\frac{1}{\nu_{\vec{w}}(2^{j+1}B)^m}
\prod_{i=1}^m\int_{2^{j+1}B}\Phi\bigg(\frac{|f_i(y_i)|}{\lambda}\bigg)\cdot w_i(y_i)\,dy_i\\
&\lesssim\|b\|_*\cdot\nu_{\vec{w}}(B)^{m(1-\kappa)}
\times\sum_{j=1}^\infty\frac{1}{\nu_{\vec{w}}(2^{j+1}B)^m}
\prod_{i=1}^mw_i\big(2^{j+1}B\big)\bigg\|\Phi\bigg(\frac{|f_i|}{\lambda}\bigg)\bigg\|_{L\log L(w_i),2^{j+1}B},
\end{split}
\end{equation*}
where the last inequality follows from the previous estimate \eqref{main esti1}. In view of \eqref{wanghua1} and \eqref{psi1}, the last expression is bounded by
\begin{equation*}
\begin{split}
&\|b\|_*\cdot\nu_{\vec{w}}(B)^{m(1-\kappa)}\times
\sum_{j=1}^\infty\frac{1}{\nu_{\vec{w}}(2^{j+1}B)^m}\prod_{i=1}^m\bigg\|\Phi\bigg(\frac{|f_i|}{\lambda}\bigg)\bigg\|_{(L\log L)^{1,\kappa}(w_i)}
\prod_{i=1}^m w_i(2^{j+1}B)^{\kappa}\\
&\lesssim\|b\|_*\prod_{i=1}^m\bigg\|\Phi\bigg(\frac{|f_i|}{\lambda}\bigg)\bigg\|_{(L\log L)^{1,\kappa}(w_i)}
\times\sum_{j=1}^\infty\frac{\nu_{\vec{w}}(B)^{m(1-\kappa)}}{\nu_{\vec{w}}(2^{j+1}B)^{m(1-\kappa)}}\\
&\lesssim\|b\|_*\prod_{i=1}^m\bigg\|\Phi\bigg(\frac{|f_i|}{\lambda}\bigg)\bigg\|_{(L\log L)^{1,\kappa}(w_i)}.
\end{split}
\end{equation*}
Let us return to the proof of \eqref{assertion}. Since $\nu_{\vec{w}}\in A_1$, we know that $\nu_{\vec{w}}$ belongs to the reverse H\"{o}lder class $RH_s$ for some $1<s<+\infty$(see \cite{duoand} and \cite{grafakos3}). Here the reverse H\"{o}lder class is defined in the following way: $\omega\in RH_s$, if there is a constant $C>0$ such that
\begin{equation*}
\bigg(\frac{1}{|B|}\int_B\omega(x)^s\,dx\bigg)^{1/s}\leq C\bigg(\frac{1}{|B|}\int_B\omega(x)\,dx\bigg).
\end{equation*}
A further application of H\"older's inequality leads to that
\begin{equation*}
\begin{split}
\int_B\big|b(x)-b_{B}\big|^{\frac{\,1\,}{m}}\nu_{\vec{w}}(x)\,dx
&\leq|B|\bigg(\frac{1}{|B|}\int_B\big|b(x)-b_{B}\big|^{s'/m}\,dx\bigg)^{1/s'}\bigg(\frac{1}{|B|}\int_B\nu_{\vec{w}}(x)^s\,dx\bigg)^{1/s}\\
&\leq C\nu_{\vec{w}}(B)\bigg(\frac{1}{|B|}\int_B\big|b(x)-b_{B}\big|^{s'/m}\,dx\bigg)^{1/s'}.
\end{split}
\end{equation*}
Thus, there are two cases to be considered. If $s'/m<1$, then \eqref{assertion} holds by using H\"older's inequality again. If $s'/m\geq1$, then \eqref{assertion} holds by using Lemma \ref{BMO}(2). On the other hand, applying the pointwise estimates \eqref{pointwise2},\eqref{jj2yr} and Chebyshev's inequality, we have
\begin{equation*}
\begin{split}
\widetilde{J}^{\beta_1,\dots,\beta_m}_{\star\star}&\leq\frac{C}{\nu_{\vec{w}}(B)^{m\kappa}}
\times\frac{2^{m+1}}{\lambda^m}
\bigg(\int_{B}\big|T_\theta((b-b_{B})f^{\beta_1}_1,f^{\beta_2}_2,\ldots,f^{\beta_m}_m)(x)\big|^{\frac{\,1\,}{m}}
\nu_{\vec{w}}(x)\,dx\bigg)^m\\
&\leq C\cdot\nu_{\vec{w}}(B)^{m(1-\kappa)}
\sum_{j=1}^\infty\bigg(\prod_{i=2}^m\frac{1}{|2^{j+1}B|}\int_{2^{j+1}B}\frac{|f_i(y_i)|}{\lambda}\,dy_i\bigg)\\
&\times\bigg(\frac{1}{|2^{j+1}B|}\int_{2^{j+1}B}
\big|b(y_1)-b_{B}\big|\cdot\frac{|f_1(y_1)|}{\lambda}\,dy_1\bigg)\\
&\leq C\cdot\nu_{\vec{w}}(B)^{m(1-\kappa)}
\sum_{j=1}^\infty\bigg(\prod_{i=2}^m\frac{1}{|2^{j+1}B|}
\int_{2^{j+1}B}\frac{|f_i(y_i)|}{\lambda}w_i(y_i)\,dy_i\bigg)\\
&\times\bigg(\frac{1}{|2^{j+1}B|}\int_{2^{j+1}B}
\big|b(y_1)-b_{B}\big|\cdot\frac{|f_1(y_1)|}{\lambda}w_1(y_1)\,dy_1\bigg)\\
&\times\prod_{i=1}^m\left(\inf_{y_i\in 2^{j+1}B}w_i(y_i)\right)^{-1}\\
&\leq C\cdot\nu_{\vec{w}}(B)^{m(1-\kappa)}
\times\sum_{j=1}^\infty\frac{1}{\nu_{\vec{w}}(2^{j+1}B)^m}
\bigg(\prod_{i=2}^m\int_{2^{j+1}B}\frac{|f_i(y_i)|}{\lambda}w_i(y_i)\,dy_i\bigg)\\
&\times\bigg(\int_{2^{j+1}B}\big|b(y_1)-b_{B}\big|\cdot\frac{|f_1(y_1)|}{\lambda}w_1(y_1)\,dy_1\bigg),
\end{split}
\end{equation*}
where in the last inequality we have used the $A_{(1,\dots,1)}$ condition \eqref{A11}. In addition, using the fact that $t\leq \Phi(t)$ and \eqref{main esti1}, we get
\begin{equation*}
\begin{split}
&\int_{2^{j+1}B}\frac{|f_i(y_i)|}{\lambda}w_i(y_i)\,dy_i\\
&\leq\int_{2^{j+1}B}\Phi\bigg(\frac{|f_i(y_i)|}{\lambda}\bigg)\cdot w_i(y_i)\,dy_i\\
&\leq w_i\big(2^{j+1}B\big)\bigg\|\Phi\bigg(\frac{|f_i|}{\lambda}\bigg)\bigg\|_{L\log L(w_i),2^{j+1}B}.
\end{split}
\end{equation*}
Using the fact that $t\leq \Phi(t)$ and the previous estimate \eqref{Wholder}, we thus obtain
\begin{equation*}
\begin{split}
&\int_{2^{j+1}B}\big|b(y_1)-b_{B}\big|\cdot\frac{|f_1(y_1)|}{\lambda}w_1(y_1)\,dy_1\\
&\leq\int_{2^{j+1}B}\big|b(y_1)-b_{B}\big|\cdot\Phi\bigg(\frac{|f_1(y_1)|}{\lambda}\bigg)w_1(y_1)\,dy_1\\
&\leq C\cdot w_1\big(2^{j+1}B\big)
\big\|b-b_{B}\big\|_{\exp L(w_1),2^{j+1}B}
\bigg\|\Phi\bigg(\frac{|f_1|}{\lambda}\bigg)\bigg\|_{L\log L(w_1),2^{j+1}B}.
\end{split}
\end{equation*}
Furthermore, by the inequality \eqref{Jensen},
\begin{equation*}
\begin{split}
&\int_{2^{j+1}B}\big|b(y_1)-b_{B}\big|\cdot\frac{|f_1(y_1)|}{\lambda}w_1(y_1)\,dy_1\\
&\leq C(j+1)\|b\|_\ast\cdot w_1\big(2^{j+1}B\big)
\bigg\|\Phi\bigg(\frac{|f_1|}{\lambda}\bigg)\bigg\|_{L\log L(w_1),2^{j+1}B}.
\end{split}
\end{equation*}
Consequently, from the above two estimates, it follows that
\begin{align}\label{WJ3yr}
\widetilde{J}^{\beta_1,\dots,\beta_m}_{\star\star}
&\lesssim\|b\|_*\cdot\nu_{\vec{w}}(B)^{m(1-\kappa)}\notag\\
&\times\sum_{j=1}^\infty(j+1)\frac{1}{\nu_{\vec{w}}(2^{j+1}B)^m}
\prod_{i=1}^mw_i\big(2^{j+1}B\big)\bigg\|\Phi\bigg(\frac{|f_i|}{\lambda}\bigg)\bigg\|_{L\log L(w_i),2^{j+1}B}\notag\\
&\lesssim\|b\|_*\cdot\nu_{\vec{w}}(B)^{m(1-\kappa)}\notag\\
&\times
\sum_{j=1}^\infty(j+1)\frac{1}{\nu_{\vec{w}}(2^{j+1}B)^m}\prod_{i=1}^m\bigg\|\Phi\bigg(\frac{|f_i|}{\lambda}\bigg)\bigg\|_{(L\log L)^{1,\kappa}(w_i)}
\prod_{i=1}^m w_i(2^{j+1}B)^{\kappa}\notag\\
&\lesssim\|b\|_*\prod_{i=1}^m\bigg\|\Phi\bigg(\frac{|f_i|}{\lambda}\bigg)\bigg\|_{(L\log L)^{1,\kappa}(w_i)}
\times\sum_{j=1}^\infty(j+1)\frac{\nu_{\vec{w}}(B)^{m(1-\kappa)}}{\nu_{\vec{w}}(2^{j+1}B)^{m(1-\kappa)}}\notag\\
&\lesssim\|b\|_*\prod_{i=1}^m\bigg\|\Phi\bigg(\frac{|f_i|}{\lambda}\bigg)\bigg\|_{(L\log L)^{1,\kappa}(w_i)}.
\end{align}
where the last two inequalities follow from \eqref{wanghua1} and \eqref{compare}. This completes the proof of Theorem \ref{mainthm:4}.
\end{proof}

For the iterated commutator $\big[\Pi\vec{b},T_\theta\big]$, we can also establish the following results in the same manner as in Theorems \ref{mainthm:3} and \ref{mainthm:4}. The proof then needs appropriate but minor modifications and we leave the details to the reader.
\begin{theorem}\label{mainthm:5}
Let $m\geq2$ and $\big[\Pi\vec{b},T_\theta\big]$ be the iterated commutator of $\theta$-type Calder\'on--Zygmund operator $T_{\theta}$ with $\theta$ satisfying the condition \eqref{theta1} and $\vec{b}\in \mathrm{BMO}^m$. If $1<p_1,\dots,p_m<+\infty$ and $1/m<p<+\infty$ with $1/p=\sum_{i=1}^m 1/{p_i}$, and $\vec{w}=(w_1,\ldots,w_m)\in A_{\vec{P}}$ with $w_1,\ldots,w_m\in A_\infty$, then for any $0<\kappa<1$, the iterated commutator $\big[\Pi\vec{b},T_\theta\big]$ is bounded from $L^{p_1,\kappa}(w_1)\times L^{p_2,\kappa}(w_2)\times\cdots
\times L^{p_m,\kappa}(w_m)$ into $L^{p,\kappa}(\nu_{\vec{w}})$ with $\nu_{\vec{w}}=\prod_{i=1}^m w_i^{p/{p_i}}$.
\end{theorem}

\begin{theorem}\label{mainthm:6}
Let $m\geq2$ and $\big[\Pi\vec{b},T_\theta\big]$ be the iterated commutator of $\theta$-type Calder\'on--Zygmund operator $T_{\theta}$ with $\theta$ satisfying the condition \eqref{theta3} and $\vec{b}\in \mathrm{BMO}^m$. Assume that $\vec{w}=(w_1,\ldots,w_m)\in A_{(1,\dots,1)}$ with $w_1,\ldots,w_m\in A_\infty$. If $p_i=1$, $i=1,2,\ldots,m$ and $p=1/m$, then for any given $\lambda>0$ and any ball $B\subset\mathbb R^n$, there exists a constant $C>0$ such that
\begin{equation*}
\begin{split}
&\frac{1}{\nu_{\vec{w}}(B)^{m\kappa}}\cdot\Big[\nu_{\vec{w}}\Big(\Big\{x\in B:\big|\big[\Pi\vec{b},T_\theta\big](\vec{f})(x)\big|>\lambda^m\Big\}\Big)\Big]^m\\
&\leq C\cdot
\prod_{i=1}^m\bigg\|\Phi^{(m)}\bigg(\frac{|f_i|}{\lambda}\bigg)\bigg\|_{(L\log L)^{1,\kappa}(w_i)},
\end{split}
\end{equation*}
where $\nu_{\vec{w}}=\prod_{i=1}^m w_i^{1/{m}}$, $\Phi(t)=t\cdot(1+\log^+t)$ and $\Phi^{(m)}=\overbrace{\Phi\circ\cdots\circ\Phi}^m$.
\end{theorem}

Finally, in view of the relation \eqref{include}, we have the following results.
\begin{corollary}
Let $m\geq2$ and $\vec{b}\in \mathrm{BMO}^m$. If $1<p_1,\dots,p_m<+\infty$ and $1/m<p<+\infty$ with $1/p=\sum_{i=1}^m 1/{p_i}$, and $\vec{w}=(w_1,\ldots,w_m)\in\prod_{i=1}^m A_{p_i}$, then for any $0<\kappa<1$, both the multilinear commutator $\big[\Sigma\vec{b},T_\theta\big]$ and the iterated commutator $\big[\Pi\vec{b},T_\theta\big]$ are bounded from $L^{p_1,\kappa}(w_1)\times L^{p_2,\kappa}(w_2)\times\cdots
\times L^{p_m,\kappa}(w_m)$ into $L^{p,\kappa}(\nu_{\vec{w}})$ with $\nu_{\vec{w}}=\prod_{i=1}^m w_i^{p/{p_i}}$, provided that $\theta$ satisfies the condition \eqref{theta1}.
\end{corollary}

\begin{corollary}
Let $m\geq2$ and $\vec{b}\in \mathrm{BMO}^m$. Assume that $\vec{w}=(w_1,\ldots,w_m)\in \prod_{i=1}^m A_1$. If $p_i=1$, $i=1,2,\ldots,m$ and $p=1/m$, then for any given $\lambda>0$ and any ball $B\subset\mathbb R^n$, there exists a constant $C>0$ such that $(\nu_{\vec{w}}=\prod_{i=1}^m w_i^{1/{m}})$
\begin{equation*}
\begin{split}
&\frac{1}{\nu_{\vec{w}}(B)^{m\kappa}}\cdot\Big[\nu_{\vec{w}}\Big(\Big\{x\in B:\big|\big[\Sigma\vec{b},T_\theta\big](\vec{f})(x)\big|>\lambda^m\Big\}\Big)\Big]^m\\
&\leq C\cdot
\prod_{i=1}^m\bigg\|\Phi\bigg(\frac{|f_i|}{\lambda}\bigg)\bigg\|_{(L\log L)^{1,\kappa}(w_i)},
\end{split}
\end{equation*}
provided that $\theta$ satisfies the condition \eqref{theta2}, and
\begin{equation*}
\begin{split}
&\frac{1}{\nu_{\vec{w}}(B)^{m\kappa}}\cdot\Big[\nu_{\vec{w}}\Big(\Big\{x\in B:\big|\big[\Pi\vec{b},T_\theta\big](\vec{f})(x)\big|>\lambda^m\Big\}\Big)\Big]^m\\
&\leq C\cdot
\prod_{i=1}^m\bigg\|\Phi^{(m)}\bigg(\frac{|f_i|}{\lambda}\bigg)\bigg\|_{(L\log L)^{1,\kappa}(w_i)},
\end{split}
\end{equation*}
provided that $\theta$ satisfies the condition \eqref{theta3}.
\end{corollary}

\section{Appendix}
As pointed out in Remark 1.6, the conclusion of Theorem \ref{comm} also holds with \eqref{theta2} replaced by the weaker condition \eqref{theta1}. In the last section, we shall give the proof of Theorem \ref{commwh} since the proof of Theorem \ref{comm} is quite similar and easier. Let $\vec{b}\in\mathrm{BMO}^m$ and $\big[\Pi\vec{b},T_\theta\big]$ be the iterated commutator of $\theta$-type Calder\'on--Zygmund operator $T_{\theta}$ with $\theta$ satisfying the condition \eqref{theta1}. We prove that if $\vec{w}=(w_1,\ldots,w_m)\in A_{\vec{P}}$, then there exists a constant $C>0$ independent of $\vec{f}$ such that
\begin{equation*}
\big\|\big[\Pi\vec{b},T_\theta\big](\vec{f})\big\|_{L^p(\nu_{\vec{w}})}\le C\prod_{k=1}^m\big\|f_k\big\|_{L^{p_k}(w_k)},
\end{equation*}
where $\nu_{\vec{w}}=\prod_{k=1}^m w_k^{p/{p_k}}$.
The method used here is different from the one in \cite{lu}.
The basic idea of the proof is taken from \cite{alvarez,ding} and \cite[Proposition 3.1]{perez3}.
For $b_k\in\mathrm{BMO}(\mathbb R^n)$ with $1\leq k\leq m$, we denote $F_k(\xi)=e^{\xi[b_k(x)-b_k(y)]}$, $\xi\in\mathbb C$. Then by the analyticity of $F_k(\xi)$ on $\mathbb C$ and the Cauchy integral formula, we get
\begin{equation}\label{cauchy}
\begin{split}
b_k(x)-b_k(y)&=F'_k(0)=\frac{1}{2\pi i}\int_{|\xi|=1}\frac{F_k(\xi)}{\xi^2}\,d\xi\\
&=\frac{1}{2\pi}\int_0^{2\pi}e^{e^{i \varphi_k}[b_k(x)-b_k(y)]}\cdot e^{-i \varphi_k}d\varphi_k.
\end{split}
\end{equation}
Hence, by \eqref{iteratedc} and \eqref{cauchy}, we can see that
\begin{equation*}
\begin{split}
&\big[\Pi\vec{b},T_\theta\big](f_1,\ldots,f_m)(x)\\
&=\int_{(\mathbb R^n)^m}\prod_{k=1}^m\big[b_k(x)-b_k(y_k)\big]K(x,y_1,\dots,y_m)f_1(y_1)\cdots f_m(y_m)\,dy_1\cdots dy_m\\
&=\int_{(\mathbb R^n)^m}\prod_{k=1}^m\bigg(\frac{1}{2\pi}\int_0^{2\pi}e^{e^{i\varphi_k}[b_k(x)-b_k(y_k)]}\cdot e^{-i\varphi_k}d\varphi_k\bigg)\\
&\times K(x,y_1,\dots,y_m)f_1(y_1)\cdots f_m(y_m)\,dy_1\cdots dy_m\\
&=\int_{(\mathbb R^n)^m}\bigg[\frac{1}{(2\pi)^m}\int_{[0,2\pi]^m}\bigg(\prod_{k=1}^me^{e^{i\varphi_k}b_k(x)}\cdot e^{-i\varphi_k}\bigg)d\varphi_1\cdots d\varphi_m\bigg]\\
&\times K(x,y_1,\dots,y_m)\prod_{k=1}^m e^{-e^{i\varphi_k}b_k(y_k)}\cdot f_k(y_k)\,dy_1\cdots dy_m\\
&=\frac{1}{(2\pi)^m}\int_{[0,2\pi]^m}T_{\theta}\Big(e^{-e^{i\varphi_1}b_1}\cdot f_1,\dots,e^{-e^{i\varphi_m}b_m}\cdot f_m\Big)(x)
\bigg(\prod_{k=1}^me^{e^{i\varphi_k}b_k(x)}\cdot e^{-i\varphi_k}\bigg)d\varphi_1\cdots d\varphi_m.
\end{split}
\end{equation*}
From this, it follows that
\begin{equation*}
\begin{split}
&\big|\big[\Pi\vec{b},T_\theta\big](f_1,\ldots,f_m)(x)\big|\\
\leq &\frac{1}{(2\pi)^m}\int_{[0,2\pi]^m}\Big|T_{\theta}\Big(e^{-e^{i\varphi_1}b_1}\cdot f_1,\dots,e^{-e^{i\varphi_m}b_m}\cdot f_m\Big)(x)\Big|\bigg(\prod_{k=1}^me^{\cos\varphi_k b_k(x)}\bigg)d\varphi_1\cdots d\varphi_m.
\end{split}
\end{equation*}
For any $(\varphi_1,\dots,\varphi_m)\in[0,2\pi]^m$, define $m$-tuples
\begin{equation*}
\vec{g}_{\varphi}=\big(g^1_{\varphi_1},\dots,g^m_{\varphi_m}\big),\quad\mbox{where}\;\; g^k_{\varphi_k}=e^{-e^{i\varphi_k}b_k}\cdot f_k,\;\; k=1,2,\dots,m,
\end{equation*}
and define
\begin{equation*}
\vec{w}_{\varphi}=\big(w^1_{\varphi_1},\dots,w^m_{\varphi_m}\big),\quad \mbox{where}\;\;
w^k_{\varphi_k}=w_k\cdot e^{p_k\cos\varphi_k b_k},\;\; k=1,2,\dots,m.
\end{equation*}
Set
\begin{equation*}
\nu^\ast_{\vec{w}}=\prod_{k=1}^m\big(w^k_{\varphi_k}\big)^{p/{p_k}}.
\end{equation*}
Then we have
\begin{equation*}
\nu^\ast_{\vec{w}}=\prod_{k=1}^m\big(w_k\cdot e^{p_k\cos\varphi_k b_k}\big)^{p/{p_k}}
=\nu_{\vec{w}}\cdot\prod_{k=1}^m e^{p\cos\varphi_k b_k}.
\end{equation*}
Using Minkowski's inequality, we thus obtain
\begin{equation*}
\begin{split}
\Big\|\big[\Pi\vec{b},T_\theta\big](\vec{f})\Big\|_{L^p(\nu_{\vec{w}})}
&\leq \frac{1}{(2\pi)^m}\int_{[0,2\pi]^m}\bigg\|T_{\theta}(\vec{g}_{\varphi})\prod_{k=1}^me^{\cos\varphi_k b_k}\bigg\|_{L^p(\nu_{\vec{w}})}
d\varphi_1\cdots d\varphi_m\\
&=\frac{1}{(2\pi)^m}\int_{[0,2\pi]^m}\big\|T_{\theta}(\vec{g}_{\varphi})\big\|_{L^p(\nu^\ast_{\vec{w}})}
d\varphi_1\cdots d\varphi_m.
\end{split}
\end{equation*}
Since $\vec{w}=(w_1,\ldots,w_m)\in A_{\vec{P}}$, we have $\nu_{\vec{w}}\in A_{mp}$ and $w_k^{1-p'_k}\in A_{mp'_k}$,$k=1,2,\ldots,m$, by using Lemma \ref{multi}. Hence, by the self-improvement property of $A_p$ weights (see \cite{duoand,garcia}), there exist some positive numbers $\varepsilon',\varepsilon_1,\dots,\varepsilon_m>0$ such that
\begin{equation*}
\nu_{\vec{w}}^{1+\varepsilon'}\in A_{mp}\quad \&\quad\big(w_k^{1-p'_k}\big)^{1+\varepsilon_k}\in A_{mp'_k},~k=1,2,\dots,m.
\end{equation*}
Now choose
\begin{equation*}
\varepsilon:=\min\big\{\varepsilon',\varepsilon_1,\dots,\varepsilon_m\big\}.
\end{equation*}
Then we have
\begin{equation*}
\nu_{\vec{w}}^{1+\varepsilon}\in A_{mp}\quad \&\quad\big(w_k^{1-p'_k}\big)^{1+\varepsilon}=\big(w_k^{1+\varepsilon}\big)^{1-p'_k}\in A_{mp'_k},~k=1,2,\dots,m,
\end{equation*}
which implies $(\vec{w})^{1+\varepsilon}:=(w_1^{1+\varepsilon},\ldots,w_m^{1+\varepsilon})\in A_{\vec{P}}$ by using Lemma \ref{multi} again.
Note that
\begin{equation*}
\prod_{k=1}^m\big(w_k^{1+\varepsilon}\big)^{p/{p_k}}=\bigg(\prod_{k=1}^m w_k^{p/{p_k}}\bigg)^{1+\varepsilon}
=(\nu_{\vec{w}})^{1+\varepsilon}.
\end{equation*}
Thus by Theorem \ref{strong},
\begin{equation}\label{inter1}
T_{\theta}:L^{p_1}(w_1^{1+\varepsilon})\times\cdots\times L^{p_m}(w_m^{1+\varepsilon})
\longrightarrow L^p((\nu_{\vec{w}})^{1+\varepsilon}).
\end{equation}
On the other hand, for any fixed $\eta>0$, it is known that when $b\in\mathrm{BMO}(\mathbb R^n)$ with $\|b\|_{*}<\min\{C_2/{\eta},C_2(p-1)/{\eta}\}$, where $C_2$ is the constant in the John--Nirenberg inequality mentioned above, we have $e^{\eta b(x)}\in A_p$ for $1<p<+\infty$
(see \cite[Lemma 1]{ding}). For $b_k\in\mathrm{BMO}(\mathbb R^n)$($1\leq k\leq m$), we now choose
\begin{equation*}
\eta_k:=\frac{p_k(1+\varepsilon)}{\varepsilon}.
\end{equation*}
For such $\eta_k>0$, we may assume that $\|b_k\|_{*}<\min\{C_2/{\eta_k},C_2(p_k-1)/{\eta_k}\}$. The general case can be proved using the linearity of $T_{\theta}$ as well. Then for any $\varphi_k\in[0,2\pi]$, we have $\cos\varphi_k\cdot b_k(x)\in \mathrm{BMO}(\mathbb R^n)$, and
\begin{equation*}
\|\cos\varphi_k\cdot b_k\|_{*}\leq\|b_k\|_{*}<\min\big\{C_2/{\eta_k},C_2(p_k-1)/{\eta_k}\big\},
\end{equation*}
which implies that each $\nu_k(x):=e^{\eta_k\cos\varphi_kb_k(x)}\in A_{p_k}$ for $1<p_k<+\infty$, $k=1,2,\dots,m$. Notice that
\begin{equation*}
\prod_{k=1}^m\big(e^{\frac{1+\varepsilon}{\varepsilon}p_k\cos\varphi_kb_k}\big)^{p/{p_k}}=
\prod_{k=1}^m\big(e^{\frac{1+\varepsilon}{\varepsilon}p\cos\varphi_kb_k}\big)=
\bigg(\prod_{k=1}^m e^{p\cos\varphi_kb_k}\bigg)^{\frac{1+\varepsilon}{\varepsilon}}.
\end{equation*}
This fact along with \eqref{include} and Theorem \ref{strong} gives us that
\begin{equation}\label{inter2}
T_{\theta}:L^{p_1}\big(e^{\frac{1+\varepsilon}{\varepsilon}p_1\cos\varphi_1b_1}\big)
\times\cdots
\times L^{p_m}\big(e^{\frac{1+\varepsilon}{\varepsilon}p_m\cos\varphi_mb_m}\big)\longrightarrow
L^p\Big(\big(\prod_{k=1}^m e^{p\cos\varphi_kb_k}\big)^{\frac{1+\varepsilon}{\varepsilon}}\Big).
\end{equation}
Interpolating between \eqref{inter1} and \eqref{inter2}(see \cite{bergh,stein}) we obtain that
\begin{equation*}
T_{\theta}:L^{p_1}\big(w_1e^{p_1\cos\varphi_1b_1}\big)
\times\cdots
\times L^{p_m}\big(w_me^{p_m\cos\varphi_mb_m}\big)\longrightarrow
L^p\Big(\nu_{\vec{w}}\prod_{k=1}^m e^{p\cos\varphi_kb_k}\Big),
\end{equation*}
that is
\begin{equation}\label{inter3}
T_{\theta}:L^{p_1}\big(w^1_{\varphi_1}\big)
\times\cdots
\times L^{p_m}\big(w^m_{\varphi_m}\big)\longrightarrow
L^p\big(\nu^\ast_{\vec{w}}\big).
\end{equation}
By \eqref{inter3} we have
\begin{equation}
\big\|T_{\theta}(\vec{g}_{\varphi})\big\|_{L^p(\nu^\ast_{\vec{w}})}
\leq C\prod_{k=1}^m\big\|g^k_{\varphi_k}\big\|_{L^{p_k}(w^k_{\varphi_k})}.
\end{equation}
Since $f_k\in L^{p_k}(w_k)$, it is easy to check that for any $\varphi_k\in[0,2\pi]$,
\begin{equation*}
\begin{split}
\big\|g^k_{\varphi_k}\big\|_{L^{p_k}(w^k_{\varphi_k})}
&=\bigg(\int_{\mathbb R^n}\big|g^k_{\varphi_k}(x)\big|^{p_k} w_k(x)\cdot e^{p_k\cos\varphi_k b_k(x)}dx\bigg)^{1/{p_k}}\\
&=\bigg(\int_{\mathbb R^n}\big|f_k(x)\big|^{p_k}e^{-p_k\cos\varphi_k b_k(x)}\cdot w_k(x)\cdot e^{p_k\cos\varphi_k b_k(x)}dx\bigg)^{1/{p_k}}\\
&=\bigg(\int_{\mathbb R^n}\big|f_k(x)\big|^{p_k} w_k(x)dx\bigg)^{1/{p_k}}=\big\|f_k\big\|_{L^{p_k}(w_k)}.
\end{split}
\end{equation*}
Therefore
\begin{equation*}
\begin{split}
\Big\|\big[\Pi\vec{b},T_\theta\big](\vec{f})\Big\|_{L^p(\nu_{\vec{w}})}
&\leq C\frac{1}{(2\pi)^m}\int_{[0,2\pi]^m}\prod_{k=1}^m\big\|g^k_{\varphi_k}\big\|_{L^{p_k}(w^k_{\varphi_k})}
d\varphi_1\cdots d\varphi_m\\
&=C\frac{1}{(2\pi)^m}\int_{[0,2\pi]^m}\prod_{k=1}^m\big\|f_k\big\|_{L^{p_k}(w_k)}
d\varphi_1\cdots d\varphi_m\\
&\leq C\prod_{k=1}^m\big\|f_k\big\|_{L^{p_k}(w_k)},
\end{split}
\end{equation*}
which is our desired estimate.

\subsection*{Acknowledgment}
This work was supported by the Natural Science Foundation of China (Grant No. XJEDU2020Y002 and 2022D01C407).


\begin{thebibliography}{99}

\bibitem{alvarez}
J. Alvarez, R. J. Bagby, D. S. Kurtz and C. P\'erez, \emph{Weighted estimates for commutators of linear operators}, Studia Math, \textbf{104}(1993), 195--209.
\bibitem{bergh}
J. Bergh and J. L\"ofstr\"om, \emph{Interpolation Spaces}. An Introduction, Springer--Verlag, 1976.
\bibitem{ding}
Y. Ding, S. Z. Lu and K. Yabuta, \emph{On commutators of Marcinkiewicz integrals with rough kernel}, J. Math. Anal. Appl, \textbf{275}(2002), 60--68.
\bibitem{duoand}
J. Duoandikoetxea, \emph{Fourier Analysis}, American Mathematical Society, Providence, Rhode Island, 2000.
\bibitem{garcia}
J. Garcia-Cuerva and J. L. Rubio de Francia, \emph{Weighted Norm Inequalities and Related Topics}, North-Holland, Amsterdam, 1985.
\bibitem{grafakos2}
L. Grafakos, \emph{Classical Fourier Analysis}, Third Edition, Springer-Verlag, 2014
\bibitem{grafakos3}
L. Grafakos, \emph{Modern Fourier Analysis}, Third Edition, Springer-Verlag, 2014.
\bibitem{grafakos}
L. Grafakos and R. H. Torres, \emph{Multilinear Calder\'on--Zygmund theory}, Adv. Math., \textbf{165}(2002), 124--164.
\bibitem{john}
F. John and L. Nirenberg, \emph{On functions of bounded mean oscillation}, Comm. Pure Appl. Math,
    \textbf{14}(1961), 415--426.
\bibitem{komori}
Y. Komori and S. Shirai, \emph{Weighted Morrey spaces and a singular integral operator}, Math. Nachr, \textbf{282}(2009), 219--231.
\bibitem{lerner}
A. K. Lerner, S. Ombrosi, C. P\'erez, R. H. Torres and R. Trujillo-Gonz\'alez, \emph{New maximal functions and multiple weights for the multilinear Calder\'on--Zygmund theory}, Adv. Math., \textbf{220}(2009), 1222--1264.
\bibitem{lida}
T. Iida, E. Sato, Y. Sawano and H. Tanaka, \emph{Multilinear fractional integrals on Morrey spaces}, Acta Math. Sin. (Engl. Ser.), \textbf{28}(2012), 1375--1384.
\bibitem{liu}
Z. G. Liu and S. Z. Lu, \emph{Endpoint estimates for commutators of Calder\'on--Zygmund type operators}, Kodai Math. J., \textbf{25}(2002), 79--88.
\bibitem{lu}
G. Z. Lu and P. Zhang, \emph{Multilinear Calder\'on--Zygmund operators with kernels of Dini's type and applications}, Nonlinear Anal., \textbf{107}(2014), 92--117.
\bibitem{ma}
D. Maldonado and V. Naibo, \emph{Weighted norm inequalities for paraproducts and bilinear pseudodifferential operators with mild regularity}, J. Fourier Anal. Appl.,  \textbf{15}(2009), 218--261.
\bibitem{morrey}
C. B. Morrey, \emph{On the solutions of quasi-linear elliptic partial differential equations}, Trans. Amer. Math. Soc, \textbf{43}(1938), 126--166.
\bibitem{perez3}
C. P\'erez and R. H. Torres, \emph{Sharp maximal function estimates for multilinear singular integrals}, Contemp. Math., \textbf{320}(2003), 323--331.
\bibitem{perez}
C. P\'erez, G. Pradolini, R. H. Torres and R. Trujillo-Gonz\'alez, \emph{End-point estimates for iterated commutators of multilinear singular integrals}, Bull. Lond. Math. Soc., \textbf{46}(2014), 26--42.
\bibitem{quek}
T. S. Quek and D. C. Yang, \emph{Calder\'on--Zygmund-type operators on weighted weak Hardy spaces over $\mathbb R^n$}, Acta Math. Sinica (Engl. Ser), \textbf{16}(2000), 141--160.
\bibitem{rao}
M. M. Rao and Z. D. Ren, \emph{Theory of Orlicz Spaces}, Marcel Dekker, New York, 1991.
\bibitem{sa}
Y. Sawano, S. Sugano and H. Tanaka, \emph{Orlicz--Morrey spaces and fractional operators}, Potential Anal., \textbf{36}(2012), 517--556.
\bibitem{stein}
E. M. Stein and G. Weiss, \emph{Interpolation of operators with change of measures}, Trans. Amer. Math. Soc, \textbf{87}(1958), 159--172.
\bibitem{stein2}
E. M. Stein, \emph{Harmonic Analysis: Real-Variable Methods, Orthogonality, and Oscillatory Integrals}, Princeton Univ. Press, Princeton, New Jersey, 1993.
\bibitem{wang1}
H. Wang, \emph{Intrinsic square functions on the weighted Morrey spaces}, J. Math. Anal. Appl, \textbf{396}(2012), 302--314.
\bibitem{wang3}
H. Wang, \emph{Weighted inequalities for fractional integral operators and linear commutators in the Morrey-type spaces}, J. Inequal. Appl, 2017, Paper No. 6, 33 pp.
\bibitem{wang4}
H. Wang, \emph{Boundedness of $\theta$-type Calder\'on--Zygmund operators and commutators in the generalized weighted Morrey spaces}, J. Funct. Spaces, 2016, Art. ID 1309348, 18 pp.
\bibitem{yabuta}
K. Yabuta, \emph{Generalizations of Calder\'on--Zygmund operators}, Studia Math, \textbf{82}(1985), 17--31.
\bibitem{zhang}
P. Zhang, \emph{Weighted endpoint estimates for commutators of Marcinkiewicz integrals}, Acta Math. Sinica (Engl. Ser), \textbf{26}(2010), 1709--1722.
\bibitem{zhang2}
P.Zhang and H.Xu, \emph{Sharp weighted estimates for commutators of Calder\'on--Zygmund type operators}, Acta Math. Sinica(Chin. Ser), \textbf{48}(2005), 625--636.

\end{thebibliography}
\end{document}